\documentclass[11pt,reqno]{amsart}
\usepackage{amsmath,amssymb,mathrsfs}
\usepackage{graphicx,cite,cases,arydshln}
\usepackage[dvips]{epsfig,graphicx}
\usepackage{subfig}

\usepackage{algorithm}
\usepackage{algorithmic}

\makeatletter \@addtoreset{equation}{section}

\makeatother

\newtheorem{theorem}{Theorem}[section]
\newtheorem{proposition}{Proposition}[section]
\newtheorem{lemma}{Lemma}[section]

\newtheorem{definition}{Definition}[section]

\newtheorem{remark}{Remark}[section]
\newtheorem{example}{Example}[section]

\newcommand {\eq} [1] {\begin{equation}\label{#1}}
\newcommand {\en} {\end{equation}}
\newcommand {\cD}       {{\cal D}}

\newcommand {\R}        {{\mathbb R}}

\newcommand {\Rn}       {\R^n}
\newcommand {\Rm}       {\R^m}

\newcommand {\Rmn}      {\R^{m \times n}}

\newcommand {\mat}      [1] {\left[\begin{array}{#1}}
\newcommand {\rix}          {\end{array}\right]}

\newcommand {\diag}     {\mathop{\rm diag}\nolimits}

\numberwithin{equation}{section}

\def\bsa{{\boldsymbol a}}
\def\bsb{{\boldsymbol b}}
\def\bsc{{\boldsymbol c}}
\def\bsd{{\boldsymbol d}}
\def\bse{{\boldsymbol e}}
\def\bsf{{\boldsymbol f}}
\def\bsg{{\boldsymbol g}}

\def\bsq{{\boldsymbol q}}
\def\bsr{{\boldsymbol r}}

\def\bsu{{\boldsymbol u}}
\def\bsv{{\boldsymbol v}}

\def\bsx{{\boldsymbol x}}
\def\bsy{{\boldsymbol y}}
\def\bsz{{\boldsymbol z}}

\def\vect{{\sf vec}}
\newcommand{\rt}{\top}

\def\Oh{{\mathcal O}}

\def\cD{{\mathcal  D}}

\newcommand{\ccM  }{ {\mathcal M}  }
\newcommand{\ccN  }{ {\mathcal N}  }
\newcommand{\ccT  }{ {\mathcal T}  }
\newcommand{\ccK  }{ {\mathcal K}  }

\newcommand{\BE}{\begin{equation}}
\newcommand{\EE}{\end{equation}}
\makeatletter
\renewcommand{\section}{\@startsection{section}{1}{0mm} {-\baselineskip}{0.5\baselineskip}{\bf\leftline}}
\makeatother

\begin{document}
\title[]{\bf Condition numbers for the truncated total least squares problem and their estimations}

\author[]{Qing-Le Meng\textsuperscript{1}, Huai-An Diao\textsuperscript{2}, Zheng-Jian Bai\textsuperscript{3}}

\thanks{1
School of Mathematical Sciences, Xiamen University, Xiamen 361005, P.R. China. Email: qinglemeng@yahoo.com}
\thanks{2  Corresponding author. 
School of Mathematics and Statistics, Northeast Normal University, No. 5268 Renmin\\
 \indent\,\, Street, Changchun 130024, P.R. China. Email: hadiao@nenu.edu.cn}
 \thanks{3 School of Mathematical Sciences and Fujian Provincial Key Laboratory on Mathematical Modeling \& High Performance Scientific Computing,  Xiamen University, Xiamen 361005, P.R. China.  Email: zjbai@xmu.edu.cn}

\thanks{ The research of Z.-J. Bai is partially supported by the National Natural Science Foundation of China (No. 11671337) and the Fundamental Research Funds for the Central Universities (No. 20720180008).}

\maketitle
\begin{quote}
{\small {\bf Abstract.}
In this paper, we present explicit expressions for the mixed and componentwise condition numbers of the truncated total least squares (TTLS) solution of $A\boldsymbol{x} \approx \boldsymbol{b} $ under the genericity condition, where $A$ is a $m\times n$ real data matrix and $\boldsymbol{b} $ is a real $m$-vector. Moreover,  we reveal that   normwise, componentwise and mixed condition numbers for the TTLS problem can recover the previous corresponding  counterparts for the total least squares (TLS) problem when the truncated level of for the TTLS problem is $n$. When $A$ is a structured matrix,  the structured perturbations for the structured truncated TLS (STTLS) problem are investigated and  the corresponding  explicit expressions for the structured normwise, componentwise and mixed condition numbers for the STTLS problem are obtained. Furthermore, the relationships between the structured and unstructured normwise, componentwise and mixed condition numbers for the STTLS problem are studied.  Based on small sample statistical condition estimation, reliable condition estimation algorithms are proposed for both unstructured and structured normwise, mixed and componentwise cases,  which utilize  the SVD of the augmented matrix $[A~\boldsymbol{b} ]$. The proposed  condition estimation algorithms can be integrated into the SVD-based direct solver for the small and medium size  TTLS problem to give the error estimation for the numerical TTLS solution. Numerical experiments are reported to illustrate the reliability of the proposed condition estimation algorithms.}
\end{quote}
{\small{\bf Keywords:} Truncated total least squares, normwise perturbation, componentwise perturbation, structured perturbation, singular value decomposition, small sample statistical condition estimation.
}
\maketitle

\section{Introduction}


In this paper, we consider the following linear model
\BE\label{lm}
A\bsx\approx \bsb,
\EE
where the data matrix $A\in \R^{m\times n}$ and the observation  vector $\bsb\in \R^{m}$ are  both perturbed. When $m>n$, the linear model (\ref{lm}) is overdetermined.  To find a solution to  (\ref{lm}), one may solve  the following minimization problem:
\BE\label{tls}
\begin{array}{cc}
\min & \big\|[\Delta A~\Delta \bsb]\big\|_F \\[2mm]
\mbox{subject to (s.t.)} & (A+\Delta A)\bsx=\bsb+\Delta \bsb,
\end{array}
\EE
where $\|\cdot\|_F$ means the Frobenius matrix norm. This is the  classical total  least square (TLS) problem, which was originally  proposed by Golub and Van Loan \cite{GolubTLS1980}.

The TLS problem is often used for the linear model (\ref{lm}) when the augmented matrix $[A~ \bsb]$ is rank deficient, i.e., the small singular values of $[A~\bsb]$ are assumed to be separated from the others. More interestingly, the truncated total least square (TTLS) method aims to solve the linear model (\ref{lm}) in the sense that the small singular values of  $[A~ \bsb]$ are set to be zeros. For the discussion of the TTLS, one may refer to \cite[\S 3.6.1]{VanHuffelVandewalle1991Book} and \cite{fb94,fgho97}. The TTLS problem  arises in various applications such as linear system theory, computer vision, image reconstruction, system identification, speech and audio processing, modal and spectral analysis,  and astronomy, etc. The overview of the TTLS can be found in \cite{MarkovskyVanHuffel2007}.

Let $k$ be the predefined  truncated level, where $1\leq k\leq n$. The TTLS problem  aims to solve the following problem:
\begin{equation}\label{TLS:definition}
	\bsx_k=\arg\min \|\bsx\|_2, \mbox{ subject to } A_k \bsx=\bsb_k,
\end{equation}
where $\|\cdot\|_2$ denotes the Euclidean vector norm or its induced matrix norm and $[A_k~\, \bsb_k]$ is the best rank-$k$ approximation of $[A~\,\bsb]$ in the  Frobenius norm. The TTLS problem can be viewed as the regularized TLS \eqref{tls} by truncating the small singular values of $[A~\bsb]$ to be zero (cf.\cite{fb94,fgho97}).

In order to solve (\ref{TLS:definition}), we first recall  the singular value decomposition (SVD) of $[A~ \bsb]$ which is given by
\begin{equation}\label{eq:svd}
	[A~\bsb]=U \Sigma V^\top,
\end{equation}
where $U \in \R^{m\times m}$ and $V \in \R^{(n+1) \times (n+1)} $ are orthogonal matrices and $\Sigma$ is a $m\times (n+1)$ real matrix with a vector $[\sigma_1,\ldots,\sigma_p]^\top$ on its diagonal and $\sigma_1\ge\sigma_2\ge\cdots\ge\sigma_p\ge 0$, where $p=\min\{m,n+1\}$. Here, $\diag(\bsx)$ is a diagonal matrix with a vector $\bsx$ on its diagonal and the superscript
``$\cdot^\top$" takes the transpose  of a matrix or vector. Then we have $[A_k~\bsb_k]=U \Sigma_k V^\top$, where $\Sigma_k=\diag([\sigma_1,\ldots,\sigma_k,0,\ldots,0])\in\R^{m \times (n+1)}$. Suppose the truncation level $k < \min \{m, \, n+1\} $ satisfies the condition
 \begin{equation}\label{eq:generic}
 \sigma_k > \sigma_{k+1 }.
 \end{equation}
Define
\begin{equation}\label{eq:vp}
	V=\left[
	\begin{array}{cc}
                V_{11} &  V_{12} \\[2mm]
                V_{21} &V_{22}
                \end{array}
                \right],\quad V_{11}\in \R^{n\times k}.
\end{equation}
If
\begin{equation}\label{eq:ge con}
	V_{22} \neq 0,
\end{equation}
then the TTLS problem is  {\em generic} and the TTLS solution $\bsx_k$ is given by \cite{gtt13max}
\begin{equation}\label{eq:xk}
	\bsx_k = - \frac{1}{\|V_{22}\|_2^2}V_{12} V_{22}^\top.
\end{equation}

Condition numbers measure the worst-case sensitivity of an input data to small perturbations.  {Normwise} condition numbers for the TLS problem \eqref{tls} under the genericity condition were studied in \cite{Baboulin2011SIMAX,JiaLi2013}, where SVD-based explicit formulas for normwise condition numbers were derived.  The normwise condition number of the truncated SVD solution to a linear model as (\ref{lm}) was introduced  in \cite{BergouThe}.  When the data is sparse or badly scaled, it is more suitable to consider the {componentwise perturbations} since normwise perturbations only measure the perturbation for the data by means of norm and may ignore the relative size of the perturbation on its small (or zero) entries (cf. \cite{Higham2002Book}).  There are two types of condition numbers in componentwise perturbations. The mixed condition number measures the errors in the output using norms and the input perturbations componentwise, while the componentwise condition number measures  both the error in the output and the perturbation in the input componentwise (cf. \cite{Gohberg}). The Kronecker product based formulas for the mixed and componentwise condition numbers to the TLS problem \eqref{tls} were derived  in \cite{Zhou,ds18}. 
The corresponding componentwise perturbation analysis for the multidimensional TLS problem and mixed least squares-TLS problem can be found in \cite{Zheng1,Zheng2}.

Gratton et al. in \cite{gtt13max} investigated the normwise condition number for the TTLS problem \eqref{TLS:definition}.   The normwise condition number formula and its computable upper bounds for the TTLS solution \eqref{TLS:definition} were derived (cf. \cite[Theorems 2.4-2.6]{gtt13max}), which rely on the SVD of the augmented matrix $[A~\bsb]$. Since the normwise condition number formula for the TTLS problem \eqref{TLS:definition} involves Kronecker product, which is not easy to compute or evaluate even for the medium size TTLS problem, the condition estimation method based on the power method \cite{Higham2002Book} or the Golub-Kahan-Lanczos (GKL) bidiagonalization algorithm \cite{gk65} was  proposed to estimate the spectral norm of Fr\'echet derivative matrix related to \eqref{TLS:definition}. Furthermore, as point in \cite{gtt13max}, first-order perturbation bounds  based on the normwise condition number can significantly improve the pervious normwise perturbation results in \cite{Fierro1996Perturbation,Wei1992The} for \eqref{TLS:definition}. 

As mentioned before, when the TTLS problem \eqref{TLS:definition} is sparse or badly scaled, which often occurs in scientific computing, the conditioning based on normwise perturbation analysis may severely overestimate the true error of the numerical solution to \eqref{TLS:definition}. Indeed, from the numerical results for Example \ref{example:small} in Section \ref{sec:ex}, the TTLS problem \eqref{TLS:definition}  with respect to the specific data $A$ and $\bsb$ is well-conditioned under componentwise perturbation analysis while it is very ill-conditioned under normwise perturbation, which implies that the normwise relative errors for the numerical solution to \eqref{TLS:definition} are pessimistic. In this paper, we propose the mixed and componentwise condition number for the TTLS problem \eqref{TLS:definition} and the corresponding explicit expressions are derived, which can capture the true conditioning of  \eqref{TLS:definition}   with respect to the sparsity and scaling for the input data. As shown in Example \ref{example:small}, the introduced mixed and componentwise condition numbers for \eqref{TLS:definition} can be much smaller than the normwise condition number appeared in \cite{gtt13max}, which can improve the first-order perturbation bounds for \eqref{TLS:definition} significantly. Furthermore, when the truncated level $k$ in \eqref{TLS:definition} is selected  to be $n$,  \eqref{TLS:definition} reduces to \eqref{tls}. The normwised, mixed and componentwise condition numbers for  the TTLS problem \eqref{TLS:definition} are shown to be mathematically equivalent to the corresponding ones \cite{Baboulin2011SIMAX,JiaLi2013,Zhou} for  the untruncated  case from their explicit expressions.

Structured TLS problems \cite{KammNagy1998,LemmerlingVanHuffel2001,MarkovskyVanHuffel2007} had been studied extensively in the past decades. For structured TLS problems, it is suitable to investigate structured perturbations on the input data, because structure-preserving algorithms that preserve the underlying matrix structure can enhance the accuracy and efficiency of the TLS solution computation.  Structured condition numbers for structured TLS problems can be found in \cite{LiJia2011,ds18,DiaoWeiXie} and references therein. In this paper, we introduce structured perturbation analysis for the structured TTLS (STTLS) problem. The explicit structured normwise, mixed and componentwise condition numbers for the STTLS problem are obtained, and their relationships corresponding to the unstructured ones are investigated.

The Kronecker product based expressions, for both unstructured and structured normwise, mixed and componentwise condition numbers of the TTLS solution in Theorems \ref{th:mixed} and \ref{th:mixed s},  involve higher dimensions and thus prevent the efficient calculations  of these condition numbers. In practice, it is important to estimate condition numbers efficiently since  the forward error for the numerical solution can be obtained  via combining condition numbers with backward errors. In this paper, based on the small sample statistical condition estimation (SCE) \cite{KenneyLaub_SISC94}, we propose reliable condition estimation algorithms for both unstructured and structured normwise, mixed and componentwise condition numbers of the TTLS solution, which utilize the SVD of $[A~\bsb]$ to reduce the computational cost. Furthermore, the proposed condition estimation algorithms can be integrated into the SVD-based direct solver for the small or medium size TTLS problem \eqref{TLS:definition}. Therefore, one can obtain the reliable forward error estimations for the numerical TTLS solution after implementing the proposed condition estimation algorithms. The main computational cost in condition number estimations for \eqref{TLS:definition} is to evaluate the directional  derivatives with respect to the generated direction during the loops in condition number estimations algorithms. We point out that the power method \cite{Higham2002Book} for estimating  the normwise condition number in \cite{gtt13max} needs to evaluate  the directional  derivatives twice in one loop. However,  only evaluating direction derivative once   is needed in the loop of  Algorithms \ref{algo:subnorm} to \ref{algstr}. Therefore, compared with the normwise condition number estimation algorithm proposed in \cite{gtt13max}, our proposed condition number estimations algorithms in this paper are more efficient in terms of the computational complexity, which are also applicable for estimating the componentwise and structured perturbations for \eqref{TLS:definition}. For recent SCE's developments  for (structured) linear systems, linear least squares and TLS problem, we refer to \cite{KenneyLaubReese1998Linear,KenneyLaubReese1998LS,LaubXia2008,DiaoWeiXie} and references therein.





The rest of this paper is organized as follows. In Section \ref{se:pr} we review pervious perturbation  results on  the TTLS problem and derive explicit expressions of the mixed and componentwise condition numbers. The structured normwise, mixed and componentwise condition numbers are also investigated in Section \ref{se:pr}, where the relationships between the unstructured normwise, mixed and componentwise condition numbers for \eqref{TLS:definition} with the corresponding structured counterparts are investigated.  In Section \ref{sect:compar}  we establish the relationship between  normwise, componentwise and mixed condition numbers for the TTLS problem and the corresponding  counterparts for the untruncated TLS. In Section \ref{sec:SCE} we are devoted to  propose  several condition estimation algorithms for the normwise, mixed and componentwise condition numbers of the TTLS problem. Moreover, the structured condition estimation is  considered. In Section \ref{sec:ex},   numerical examples are shown to illustrate the efficiency and reliability of the proposed algorithms and report the perturbation bounds  based on the proposed condition number. Finally, some concluding remarks are drawn in the last section.

\section{Condition numbers for the TTLS problem}\label{se:pr}

In this section we review previous perturbation results on  the TTLS problem.   The explicit expressions of the mixed and componentwise condition numbers for the TTLS problem are derived. Furthermore, for the structured TTLS problem, we propose the normwise, mixed and componentwise condition numbers, where explicit formulas for the   corresponding  counterparts are derived. The relationships between the unstructured normwise, mixed and componentwise condition numbers for \eqref{TLS:definition} with the corresponding structured counterparts are investigated. We first introduce some conventional notations. 

Throughout this paper, we use the following notation. Let $\|\cdot\|_\infty$ be the vector  $\infty$-norm or its induced matrix norm.  Let $I_n$ be  the identity matrix of order $n$. Let $\bse_j$ be the $j$-th column vector of an  identity matrix of an appropriate  dimension. The superscripts ``$\cdot^{-}$"  and``$\cdot^\dag$" mean the  inverse and the Moore-Penrose inverse of a matrix respectively. The symbol ``$\boxdot$" means componentwise multiplication of two conformal dimensional matrices. For any matrix $B=(b_{ij})$, let $|B|=(|b_{ij}|)$, where $|b_{ij}|$ denote the absolute value of $b_{ij}$. For any two matrices $B,C\in\Rmn$, $|B| \leq |C|$ represents $|b_{ij}| \leq |c_{ij}|$ for all $1\le i\le m$ and $1\le j\le n$.   For any $\bsx,\bsy\in\Rn$, we define $\bsz:=\frac{\bsy}{\bsx}$ by
\[ \bsz_i =
\left\{
\begin{array}{ll}
\bsx_i/\bsy_i, & \mbox{if $\bsy_i\neq 0$},\\
0, & \mbox{if $\bsx_i=\bsy_i= 0$},\\
\infty,& \mbox{otherwise}.
\end{array}
\right.
\]
Let $\vect(B)$ be a column vector obtained by stacking the columns of $B$ on top of one another.  For a vector $\bsb\in\R^{mn}$, let $B={{\sf unvec}}(b)\in\Rmn$, where $B_{ij}=\bsb_{i+(j-1)m}$ for $i=1,\ldots, m$ and $j=1,\ldots,n$. The symbol ``$\otimes$" means the Kronecker product and $\Pi_{m,n} \in \R^{mn \times mn}$ is a permutation matrix defined by
\begin{equation}\label{eq:kron 3}
	\vect(B^\top )= \Pi_{m,n} \vect(B), \quad \forall B \in \R^{m \times n}.
\end{equation}
Given the matrices  $X \in \R^{m  \times n}$, $D \in \R^{n  \times p}$,   and $Y \in \R^{p  \times q}$, and $X_1,X_2,Y_1,Y_2$ with appropriate dimensions, we have the following propertes of the Kronecker product and vec operator \cite{Graham1981book}:
\BE\label{eq:kron1}
\left\{
\begin{array}{c}
		\vect(XDY)=(Y^\top \otimes X) \vect(D), \\[2mm]
		(X_1 \otimes X_2) (Y_1 \otimes Y_2)= (X_1 Y_1) \otimes (X_2 Y_2),\\[2mm]
	\Pi_{p,m} (Y\otimes X ) =(X \otimes Y ) \Pi_{n,q}.
\end{array}
\right.
\EE

\subsection{Preliminaries}\label{subsec:preliminary}
In this subsection, we recall the definition of absolute normwise condition number of the TTLS solution $\bsx_k$ defined by (\ref{TLS:definition}) (cf. \cite{gtt13max}). The {\em absolute}  normwise condition number of $\bsx_k$ in (\ref{TLS:definition})  is defined by
\begin{equation}\label{eq:k}
	\kappa(A,\bsb)=\lim_{\epsilon \to 0} \sup_{\| \Delta H \|_F \leq \epsilon} \frac{\left\| \psi_k ([A~ \bsb]+\Delta H )-\psi_k ([A~ \bsb]) \right\|_2  }{\| \Delta H \|_F},
	\end{equation}
where the function  $\psi_k$ is given by
\begin{equation}\label{eq:psi}
	\psi_k([A~\bsb]) \quad : \quad \R^{m\times n} \times \R^m \rightarrow \R^n\quad : \quad [A~\, \bsb]  \mapsto \bsx_k.
\end{equation}
Let the SVD of   $[A~\, \bsb] \in \R^{m\times (n+1)}$ be given by (\ref{eq:svd}). If the truncation level $k$ satisfies the conditions (\ref{eq:generic}) and (\ref{eq:ge con}), then the explicit expression of $\kappa(A,\bsb)$ is given by \cite[Theorem 2.4]{gtt13max}
\begin{equation}\label{eq:k ex}
	\kappa(A,\bsb)=  \| M_k \|_2,
\end{equation}
where
\BE\label{eq:mk}
M_k=\frac{1}{\|V_{22}\|_2^2} [I_n\quad\bsx_k] V K D^{-1} [I_k \otimes \Sigma_2^\top\quad \Sigma_1 \otimes I_{n-k+1 }] W
\EE
with
\begin{align*}
\Sigma_1&=\diag\Big([\sigma_1,\ldots, \sigma_k]\Big) \in \R^{k \times k}, \quad \Sigma_2=\diag\Big([\sigma_{k+1},\ldots, \sigma_p]\Big) \in \R^{(m-k) \times (n-k+1)},  \nonumber \\
\sigma_1 &\geq \cdots \geq \sigma_k > \sigma_{k+1} \geq \cdots\ge \sigma_p \ge 0, \quad k< p=\min \{ m,n+1\}, \nonumber \\
	 K&=\begin{bmatrix}
		(V_{22} \otimes I_k ) \Pi_{n-k+1,k} \\
		V_{21} \otimes I_{n-k+1} 	\end{bmatrix}
,  \quad D=\Sigma_1^2 \otimes I_{n-k+1}-I_k \otimes (\Sigma_2^\top \Sigma_2 ), \nonumber  \\
W&=\begin{bmatrix}
	V_1^\top \otimes U_2^\top \\
	\Pi_{n-k+1,k} (V_2^\top \otimes U_1^\top )
\end{bmatrix}, \quad U=[U_1 \quad U_2],\quad V=[V_1\quad V_2], \nonumber \\
V_1&=\begin{bmatrix}
	V_{11}\\ V_{21}
\end{bmatrix} \in \R^{(n+1) \times k}, V_2 =\begin{bmatrix}
	V_{12} \\ V_{22}
\end{bmatrix} \in \R^{(n+1) \times (n-k+1)},  \quad U_1 \in \R^{m \times k},\quad U_2 \in \R^{m \times (m-k)}.
\end{align*}

Please be noted the the dimension of $M_k$ may be large even for medium size TTLS problems. The explicit formula $\kappa(A,\bsb)$ given by \eqref{eq:k ex} involves the computation of the spectral norm of $M_k$. Hence, upper bounds for $\kappa(A,\bsb)$ is obtain in \cite[\S 2.4]{gtt13max}, which only rely on the singular values of $[A~\, \bsb]$ and $\|\bsx\|_2$. When the data is sparse or badly scaled, the normwise condition number $\kappa(A,\bsb)$ may not reveal the conditioning of \eqref{TLS:definition}, since normwise perturbations ignore the relative size of the perturbation on its small (or zero) entries. Therefore, it is more suitable to consider the componentwise perturbation analysis for \eqref{TLS:definition} when the data is sparse or badly scaled. In the next subsection, we shall introduce the mixed and componentwise condition number for \eqref{TLS:definition}.

In \cite[\S 2.3]{gtt13max}, if both $\bsx_k$ and the full SVD of $[A\,\,\bsb]$ are available, then one may compute $\| M_k \|_2$  by using the power method \cite[Chap. 15]{Higham2002Book}  to $M_k$ or the Golub-Kahan-Lanczos (GKL)  bidiagonalization algorithm \cite{gk65} to $M_k$, where only  the matrix-vector product is needed.  However, as pointed in the introduction part,  the normwise condition number estimation algorithm in \cite{gtt13max} are devised based on the power method \cite{Higham2002Book}, which needs to evaluate the matrix-vector products $M_k \bsf$ and $M_k^\top \bsg$ in one loop for some suitable dimensional  vectors $\bsf$ and $\bsg$. In Section \ref{sec:SCE}, SCE-based condition estimation algorithms for \eqref{TLS:definition} shall be proposed, where in one loop we only need to compute the directional derivative $M_k \bsf$ but the matrix-vector product $M_k^\top \bsg$ is not involved. Therefore, compared with  normwise condition number estimation algorithm in \cite{gtt13max}, SCE-based condition estimation algorithms in Section \ref{sec:SCE}  are more efficient. 

\subsection{Mixed and componentwise condition numbers}\label{subsec:com and mix}
In Lemma \ref{lem:1} below, the first order perturbation expansion of $\psi_k$ with respect to the perturbations of the data $A$ and $\bsb$ is reviewed, which involves the Kronecker product. In order to avoid forming Kronecker product explicitly  in the explicit expression  for  the directional derivative of $\psi_k$, we derive the corresponding  equivalent formula \eqref{eq:dir derivative x} in Lemma  \ref{pro1}. Furthermore,   the directional derivative \eqref{eq:dir derivative x} can be used to save computation memory of SCE-based condition estimation algorithms in Section \ref{sec:SCE}.

%
%
%

\begin{lemma}\cite[Theorem 2.4]{gtt13max}\label{lem:1}
Let the SVD of  the augmented matrix $[A~\bsb] \in \R^{m\times (n+1)}$ be given by (\ref{eq:svd}).
Suppose $k$ is a truncation level such that $V_{22} \neq 0$ and $\sigma_k > \sigma_{k+1}$. If $[\tilde A ~\tilde \bsb  ]= [A~ \bsb]+\Delta H$ with $\|\Delta H \|_F$ sufficiently small, then, for the TTLS solution $\bsx_k$ of $A\bsx \approx \bsb$ and the TTLS solution $\tilde \bsx_{k}$ of $\tilde A \bsx \approx \tilde \bsb$, we have
	$$
	\tilde \bsx_k =\bsx_k + M_k \,\vect(\Delta H)+ \Oh(\|\Delta H \|_F^2).
	$$
	\end{lemma}

\begin{lemma}\label{pro1}
Under the same assumptions as in Lemma \ref{lem:1}, if $[\tilde A ~\tilde \bsb  ]= [A~ \bsb]+[\Delta A~\Delta \bsb]\equiv [A~\bsb]+\Delta H$ with $\|\Delta H \|_F$ sufficiently small, then the directional derivative of $\bsx_k$ at $[A~\bsb]$ in the direction $[\Delta A~ \Delta \bsb]$ is given by
\begin{align}\label{eq:dir derivative x}
 \psi_k' ([A~ \bsb];[\Delta A~ \Delta \bsb]) &= \frac{1}{\|V_{22}\|_2^2}\left(V_{11} \, (Z_1^\top + Z_2)V_{22}^\top +V_{12} \, (Z_1+Z_2^\top)V_{21}^\top  + \bsx_k\sum_{j=1}^4 c_j \right),
\end{align}
where
\begin{align*}
	Z_1&=\left(    \Sigma_2^\top U_2^\top \Delta H V_1  \right)  \boxdot \cD   \in   \R^{(n-k+1)\times k}, \quad Z_2 = \left(\Sigma_1^\top  U_1^\top  \Delta H V_2 \right) \boxdot \cD^\top \in \R^{k\times (n-k+1)},\\
c_1&=V_{21} Z_1^\top V_{22}^\top , \quad c_2 = V_{21}Z_2V_{22}^\top ,\quad c_3=V_{22}Z_1V_{21}^\top,\quad c_4= V_{22} Z_2^\top V_{21}^\top,
	\end{align*}
$\cD=[\cD(:,1),\ldots,\cD(:,k)]\in\R^{(n-k+1)\times k}$ with
\BE\label{eq:dinvi}
\cD(:,i)=\left\{\begin{array}{ll}
		\begin{bmatrix}
	(\sigma_i^2-\sigma_{k+1}^2)^{-1}\\
	\vdots \\
	(\sigma_i^2-\sigma_{m}^2)^{-1}\\
	\sigma_i^{-2} \\
	\vdots \\
	\sigma_i^{-2} \\
\end{bmatrix} \in \R^{(n-k+1)}, \quad \mbox{ if $m< n+1$}, \\[16mm]
\begin{bmatrix}
	(\sigma_i^2-\sigma_{k+1}^2)^{-1} \\
	\vdots \\
	(\sigma_i^2-\sigma_{n+1}^2)^{-1}
\end{bmatrix} \in \R^{(n-k+1)},
\quad \mbox{ if $m\geq  n+1$}.
	\end{array}\right.
\EE

\end{lemma}

\begin{proof}
From Lemma \ref{lem:1} we have
	\[
	\psi_k' ([A~ \bsb];[\Delta A~ \Delta \bsb])=M_k \vect(\Delta H),
	\]
	where $M_k$ is defined by \eqref{eq:mk}. Using \eqref{eq:kron1},  it is easy to verify that
	\begin{align}
		[I_k \otimes \Sigma_2^\top\quad\Sigma_1 \otimes I_{n-k+1 }]\, W&= (I_k \otimes \Sigma_2^\top  ) \, (V_1^\top \otimes U_2^\top  )+ (\Sigma_1^\top \otimes I_{n-k+1} ) \, \Pi_{n-k+1,k} (V_2^\top \otimes U_1^\top  ) \notag \\
		&=V_1^\top \otimes(\Sigma_2^\top  U_2^\top  ) + (\Sigma_1^\top  U_1^\top \otimes V_2^\top ) \Pi_{m,n+1}, \label{eq:p2.1 1}
	\end{align}
	 Using the fact that
	$
	\vect(\Delta H)= [\vect(\Delta A)^\top \vect(\Delta \bsb)^\top]^\top
			$
and  \eqref{eq:kron1} we have
\begin{eqnarray}\label{eq:p2.1 2}
&&[I_k \otimes \Sigma_2^\top \quad \Sigma_1 \otimes I_{n-k+1 }] \,W \vect(\Delta H) \nonumber\\
&=&\left(V_1^\top \otimes(\Sigma_2^\top  U_2^\top  ) \right )\vect(\Delta H) + (\Sigma_1^\top  U_1^\top \otimes V_2^\top ) \Pi_{m,n+1} \,\vect(\Delta H) \nonumber \\
		&=&\vect(\Sigma_2^\top U_2^\top \Delta H V_1 )+ \vect(V_2^\top \Delta H^\top U_1 \Sigma_1 ).
\end{eqnarray}
From \eqref{eq:mk}, we see that the $i$-th diagonal block $D^{(i)}$ of $D$ is given by
\BE\label{eq:di}
	D^{(i)}=
	\left\{\begin{array}{ll}
	\diag\Big([\sigma_i^2-\sigma_{k+1}^2,\ldots,\sigma_i^2-\sigma_{m}^2,\sigma_i^2,\ldots,\sigma_i^2]^\top\Big) \in \R^{(n-k+1) \times (n-k+1) }, \quad \mbox{if $m< n+1$}, \\[2mm]
	\diag\Big([\sigma_i^2-\sigma_{k+1}^2,\ldots,\sigma_i^2-\sigma_{n+1}^2]^\top\Big)\in \R^{(n-k+1) \times (n-k+1) },
\quad \mbox{ if $m\geq  n+1$},
	\end{array}\right.
\EE
for $i=1,\ldots,k$.
By the definition of  $\cD \in \R^{(n-k+1) \times k}$ we have
\BE\label{eq:d-1}
\left\{
\begin{array}{c}
D^{-1}  \vect (\Sigma_2^\top U_2^\top \Delta H V_1 )= \vect\left( \left(\Sigma_2^\top U_2^\top \Delta H V_1 \right) \boxdot \cD \right),
\\[2mm]
D^{-1} \vect (V_2^\top \Delta H^\top U_1 \Sigma_1  )= \vect\left( \left(V_2^\top \Delta H^\top U_1 \Sigma_1 \right) \boxdot \cD \right).
\end{array}
\right.
\EE
Then, using the partition of $V$ given by \eqref{eq:vp} we have
\begin{align*}
	&\quad  [I_n ~\bsx_k] V K 	=[I_n~\bsx_k] \begin{bmatrix}
		V_{11}& V_{12} \\ V_{21} & V_{22}	\end{bmatrix} \begin{bmatrix}
		(V_{22} \otimes I_k ) \Pi_{n-k+1,k} \\
		V_{21} \otimes I_{n-k+1} 	\end{bmatrix} \\
		&=V_{11} (V_{22} \otimes I_k) \Pi_{n-k+1,k}+V_{12} (V_{21} \otimes I_{n-k+1})+\bsx_k V_{21} (V_{22} \otimes I_k ) \Pi_{n-k+1,k}+\bsx_k V_{22} (V_{21} \otimes I_{n-k+1}).
\end{align*}
This, together with \eqref{eq:p2.1 2} and \eqref{eq:d-1}, yields
\begin{eqnarray*}
		&& [I_n ~\bsx_k] V K D^{-1} [I_k \otimes \Sigma_2^\top \quad \Sigma_1 \otimes I_{n-k+1 }] W \vect(\Delta H)\\
		&=&\Big(V_{11} (V_{22} \otimes I_k) \Pi_{n-k+1,k}+V_{12} (V_{21} \otimes I_{n-k+1})+\bsx_k V_{21} (V_{22} \otimes I_k ) \Pi_{n-k+1,k}+\bsx_k V_{22} (V_{21} \otimes I_{n-k+1}) \Big)\\
		&&\Big(\vect\left( (\Sigma_2^\top U_2^\top \Delta H V_1 ) \boxdot \cD \right)+ \vect\left( (V_2^\top \Delta H^\top U_1 \Sigma_1 ) \boxdot \cD \right)\Big)  \\
		&=& V_{11} \left( \left( V_1^\top \Delta H^\top U_2 \Sigma_2\right)  \boxdot \cD^\top \right) V_{22}^\top +V_{11}  \left( \left(\Sigma_1^\top  U_1^\top  \Delta H V_2 \right) \boxdot \cD^\top \right) V_{22}^\top\\
		&& +V_{12} \left( \left(    \Sigma_2^\top U_2^\top \Delta H V_1  \right)  \boxdot \cD \right ]   V_{21}^\top + V_{12} \left(\left(    V_2^\top  \Delta H^\top U_1 \Sigma_1   \right)  \boxdot \cD \right) V_{21}^\top\\
		&&+\bsx_k V_{21} \left( \left( V_1^\top \Delta H^\top U_2 \Sigma_2\right)  \boxdot \cD^\top \right) V_{22}^\top +\bsx_k V_{21}  \left( \left(\Sigma_1^\top  U_1^\top  \Delta H V_2 \right) \boxdot \cD^\top \right) V_{22}^\top\\
 &&+ \bsx_k V_{22} \left( \left(    \Sigma_2^\top U_2^\top \Delta H V_1  \right)  \boxdot \cD \right ]   V_{21}^\top+ \bsx_k V_{22} \left(\left(    V_2^\top  \Delta H^\top U_1 \Sigma_1   \right)  \boxdot \cD \right) V_{21}^\top.
	\end{eqnarray*}
This completes the proof. \qed
	\end{proof}

	When the data is sparse or badly-scaled, it is more suitable to adopt the componentwise perturbation analysis to investigate the conditioning of the TTLS problem. In the following definition, we introduce the relative {\em  mixed } and {\em componentwise} condition numbers for the TTLS problem.
	
	\begin{definition}\label{def1}
		Suppose the truncation level $k$ is chosen such that  $V_{22} \neq 0$ and $\sigma_{k}> \sigma_{k+1}$.  The { mixed } and { componentwise} condition numbers for the TTLS problem \eqref{TLS:definition} are defined as follows:
		\begin{align*}
			m(A,\bsb)&=\lim_{\epsilon \to 0} \sup_{ \left | \Delta H \right  | \leq \epsilon\big | [A\,\bsb ] \big |} \frac{\left\| \psi_k ([A~\, \bsb]+\Delta H )-\psi_k ([A~\, \bsb]) \right\|_\infty   }{ \epsilon \|\bsx_k \|_\infty},\\
			c(A,\bsb)&=\lim_{\epsilon \to 0} \sup_{ \left | \Delta H \right  | \leq \epsilon\big  |[A~\, \bsb] \big |} \frac{1}{\epsilon }\left\|  \frac{ \psi_k ([A~\, \bsb]+\Delta H )-\psi_k ([A~\, \bsb])   }{\bsx_k } \right\|_\infty.
		\end{align*}
	\end{definition}
	
	In the following theorem, we give the Kronecker product based explicit expressions of $m(A,\bsb)$ and $c(A,\bsb)$.

	\begin{theorem}\label{th:mixed}
	Suppose the truncation level $k$ is chosen such that  $V_{22} \neq 0$ and $\sigma_{k}> \sigma_{k+1}$.   Then the mixed and componentwise condition numbers $m(A,\bsb)$ and $c(A,\bsb)$ defined in Definition {\rm \ref{def1}}  for the TTLS problem \eqref{TLS:definition} can be characterized by
	\begin{subequations}
		\begin{align}
			m(A,\bsb)&= \frac{\Big\| |M_k| \vect(\,[|A|~|\bsb|]\,) \Big\|_\infty  }{ \|\bsx_k \|_\infty},\label{eq:213a} \\
			c(A,\bsb)&=\left\|  \frac{|M_k| \vect(\,[|A|~|\bsb|]\,)  }{\bsx_k } \right \|_\infty.\label{eq:213b}
		\end{align}
	\end{subequations}
	\end{theorem}

	\begin{proof} Let $\Delta H=[\Delta A~ \, \Delta \bsb]$, where $\Delta A \in \R^{m \times n}$ and $\Delta \bsb\in \R^m $. For any
	$\epsilon>0$, it follows from  $\left | \Delta H \right  | \leq \epsilon\big  |\left [A~\, \bsb \right] \big |$ that
	\[
	|\Delta A | \leq \epsilon |A|\quad\mbox{and}\quad |\Delta \bsb | \leq \epsilon |\bsb|.
	\]
	Define
	\[
	 \Theta_A=\diag(\vect(A))\quad\mbox{and}\quad \Theta_b=\diag(\bsb).
	\]
	By Lemma \ref{lem:1} we have for $\epsilon>0$ sufficiently small,
	\begin{eqnarray}\label{eq:dpsi}
	&&	 \psi_k ([A~\bsb]+\Delta H )-\psi_k ([A~\bsb]) = M_k \, \vect(\Delta H) +\Oh(\|\Delta H\|_F^2) \\
	&=& M_k  \begin{bmatrix} \Theta_A & \\[2mm]  & \Theta_b\end{bmatrix} \begin{bmatrix} \Theta_A^\dag\vect(\Delta A) \nonumber\\[2mm] \Theta_b^\dag\vect(\Delta \bsb)\end{bmatrix}  +\Oh(\|[\Delta A~\Delta \bsb]\|_F^2),
	\end{eqnarray}
and taking infinity  norms we have
	\begin{eqnarray*}
		\left\| \psi_k ([A~\, \bsb]+\Delta H )-\psi_k ([A~\, \bsb]) \right\|_\infty &=&\left\| M_k \begin{bmatrix} \Theta_A & \\[2mm]  & \Theta_b\end{bmatrix} \begin{bmatrix} \Theta_A^\dag\vect(\Delta A) \\[2mm] \Theta_b^\dag\vect(\Delta \bsb)\end{bmatrix}  \right\|_\infty +\Oh(\|[\Delta A~\Delta \bsb]\|_F^2)\\
		&\leq& \epsilon\Big\| \big| M_k \big|  \begin{bmatrix} |\Theta_A| & \\[2mm]  & |\Theta_b|\end{bmatrix} \Big\|_\infty +\Oh(\epsilon^2),
	\end{eqnarray*}
	where  the fact that $\Oh(\|[\Delta A~\Delta \bsb]\|_F^2)\le \Oh(\epsilon^2)$ is used.
	Thus,
	\begin{eqnarray*}
		m(A,\bsb)&\le &\frac{\left\| |M_k|  \begin{bmatrix} |\Theta_A| & \\[2mm]  & |\Theta_b|\end{bmatrix} \right\|_\infty   }{ \|\bsx_k \|_\infty}=\frac{\left\| |M_k| \begin{bmatrix} |\Theta_A| & \\[2mm]  & |\Theta_b|\end{bmatrix} {\bf 1}_{mn+m}\right\|_\infty  }{ \|\bsx_k \|_\infty}\\
		&=&\frac{\left\| |M_k| \begin{bmatrix}
			\vect(|A|)\\
			|\bsb|
		\end{bmatrix}\right\|_\infty  }{ \|\bsx_k \|_\infty}
		= \frac{\bigg\| |M_k| \vect\left([|A|~|\bsb|]\right) \bigg\|_\infty  }{ \|\bsx_k \|_\infty},
	\end{eqnarray*}
	where  $ {\bf 1}_{mn+m}=[1,\ldots,1]^{\top} \in \R^{mn+m}$.
	
    On the other hand, let the index $a$ is such that
    \[
    \big\| |M_k| \vect\left([|A|~|\bsb|]\right) \big\|_\infty = |M_k(a,:)| \vect([|A|~|\bsb|]),
    \]
    where $|M_k(a,:)|$ denotes the $a$-th row of $|M_k|$.
    We choose
    \[
    \vect(\Delta H) = \epsilon\, \Theta\;  \vect([|A|\,\,|\bsb|]),
\]
    where $\Theta\in\R^{mn\times mn}$ is a diagonal matrix such that $\theta_{jj}$=sign$((M_k)_{aj})$ for $j=1,2,\ldots,m(n+1)$.  Using \eqref{eq:dpsi} we have
\begin{eqnarray*}
    m(A,\bsb) &\ge&   \lim_{\epsilon \to 0}  \frac{\left\| \epsilon M_k \Theta\;  \vect([|A|\,\,|\bsb|]) +\Oh(\epsilon \|\vect([|A|\,\,|\bsb|])\|_2^2)\right\|_\infty   }{ \epsilon \|\bsx_k \|_\infty}\\
    &=&   \frac{\left\| M_k \Theta\;  \vect([|A|\,\,|\bsb|]) \right\|_\infty   }{\|\bsx_k \|_\infty}\\
    &=&  \frac{\left\| |M_k| \vect([|A|\,\,|\bsb|]) \right\|_\infty  }{ \|\bsx_k \|_\infty}.
    \end{eqnarray*}
Therefore, we derive \eqref{eq:213a}. One can use the similar argument to obtain \eqref{eq:213b}.
\qed
	\end{proof}
	
	\begin{remark}\label{remark:1}
		Based on \eqref{eq:k} and  \eqref{eq:k ex}, the relative normwise condition number for the TTLS problem \eqref{TLS:definition} can be defined and has the following expression
		\begin{equation}\label{eq:k rel}
			\kappa^{\rm rel}  (A,\bsb)=\lim_{\epsilon \to 0} \sup_{\| \Delta H \|_F \leq \epsilon \big\|[A\,\, \bsb]\big\|_F} \frac{\left\| \psi_k ([A\,\, \bsb]+\Delta H )-\psi_k ([A\,\, \bsb]) \right\|_2  }{\epsilon \|\bsx_k\|_2}=\frac{ \|M_k\|_2~ \| [A\,\,\bsb]\|_F }{\|\bsx_k\|_2 }.
		\end{equation}
	\em Using the fact that
        	$$
       \left\{\begin{array}{ll}
       \big\|\,|M_k|\,\big\|_2=\|M_k\|_2,\\[2mm]
	\| M_k\|_\infty \leq \sqrt{m(n+1)}\,\| M_k\|_2, \\[2mm]
\|\bsx_k\|_2 \leq \sqrt n  \|\bsx_k\|_\infty,\\[2mm]
\big\|\vect([A~ \bsb])\big\|_\infty \leq   \big\| [A~ \bsb]\big\|_F,
	\end{array}\right.
$$
   it is easy to see that
		\begin{equation}\label{eq:realtion}
		m(A,\bsb) \leq \sqrt{(n+1)nm} ~ \kappa^{\rm rel}  (A,\bsb).
		\end{equation}
From Example \ref{example:small}, we can see $m(A,\bsb) $ and $c(A,\bsb) $ can be much smaller than $\kappa^{\rm rel}  (A,\bsb)$ when the data is sparse and badly scaled. Therefore, one should adopt the mixed and componentwise condition number to measure the conditioning of \eqref{TLS:definition} instead of the normwise condition number when $[A~\bsb]$ is spare or badly scaled. However, since the explicit expressions of $m(A,\bsb) $ and $c(A,\bsb) $ are based on Kronecker product, which involves large dimensional computer memory to form them explicitly even for medium size TLS problems, it is necessary to propose efficient and reliable  condition estimations for $m(A,\bsb)$ and $c(A,\bsb)$, which will be investigated in Section \ref{sec:SCE}.
	\end{remark}


In  \cite{BeckSIAM2005,KammNagy1998,LemmerlingVanHuffel2001}, the structured TLS (STTLS)  problem has been studied extensively. Hence, it is interesting to study the structured perturbation analysis for the STTLS  problem.   In the following, we propose the structured normwise, mixed and componentwise condition numbers for  the STTLS  problem, where $A$ is a linear structured data matrix. Assume that ${\mathcal S} \subset \R^{m\times n}$ is a linear subspace which consists of a class of basis matrices. Suppose there are $t$ ($t\leq  mn$) linearly independent matrices $S_1,\ldots,S_t$ in $\mathcal S$, where $S_i$ are matrices of constants, typically 0's and 1's. For any $ A \in \mathcal S$, there is a uniques vector $\bsa=[a_1,\ldots,a_t]^\top\in\R^t$  such that
\begin{equation}\label{eq:A linear}
	A=\sum_{i=1}^t a_i S_i.
\end{equation}
In the following, we study the sensitivity of the STTLS solution $\bsx_k$ to perturbations on the data $\bsa$ and $\bsb$, which is defined by
\begin{align}\label{eq:g dfn s}
\psi_{s,k} (\bsa,\, \bsb) \quad : \quad \R^{t} \times
\R^m  \rightarrow \R^n  \quad
: \quad (\bsa,\, \bsb)  \mapsto \bsx_k ,
\end{align}
where $\bsx_k$ is the unique solution to the STTLS problem \eqref{TLS:definition} and   \eqref{eq:A linear}.
\begin{definition}\label{def2}
		Suppose the truncation level $k$ is chosen such that  $V_{22} \neq 0$ and $\sigma_{k}> \sigma_{k+1}$.   The absolute structured normwise, mixed  and  componentwise condition number for the STTLS problem \eqref{TLS:definition} and   \eqref{eq:A linear} are defined as follows:
		\begin{align*}
		\kappa_s(\bsa,\bsb)&=\lim_{\epsilon \to 0} \sup_{ \left\| \begin{bmatrix} \Delta \bsa\\  \Delta \bsb \end{bmatrix} \right\|_2  \leq ~ \epsilon }  \frac{\left\| \psi_{s,k} ((\bsa,\, \bsb)+(\Delta \bsa,\, \Delta \bsb) )-\psi_{s,k} (\bsa,\, \bsb) \right\|_2   }{ \left\| \begin{bmatrix} \Delta \bsa\\  \Delta \bsb \end{bmatrix} \right\|_2 },\\
			m_s(\bsa,\bsb)&=\lim_{\epsilon \to 0} \sup_{ |\Delta \bsa| \leq \epsilon |\bsa|\atop \left | \Delta \bsb\right  | \leq \epsilon\left  | \bsb  \right |} \frac{\left\| \psi_{s,k} ((\bsa,\, \bsb)+(\Delta \bsa,\, \Delta \bsb) )-\psi_{s,k} (\bsa,\, \bsb) \right\|_\infty   }{ \epsilon \|\bsx_k \|_\infty},\\
			c_s(\bsa,\bsb)&=\lim_{\epsilon \to 0} \sup_{ |\Delta \bsa| \leq \epsilon |\bsa|\,\atop \left | \Delta \bsb\right  | \leq \epsilon\left  | \bsb  \right | } \frac{1}{\epsilon }\left\|  \frac{ \psi_{s,k} ((\bsa,\, \bsb)+(\Delta \bsa,\, \Delta \bsb) )-\psi_{s,k}  (\bsa,\, \bsb)   }{\bsx_k } \right\|_\infty .
		\end{align*}
	\end{definition}

In the following lemma, we provide the first order expansion of the STTLS solution $\bsx_k$ with respect to the structured perturbations $\Delta \bsa$ on $\bsa$ and $\Delta \bsb$ on $\bsb$, which help us to derive the structured condition number expressions for the STTLS problem \eqref{TLS:definition} and  \eqref{eq:A linear}. In view of the fact that $\vect(\Delta A)= \sum_{i=1}^t \Delta a_i  \vect(S_i) $, we can prove the following lemma from Lemma  \ref{lem:1}. The detailed proof is omitted here.

\begin{lemma}\label{lem:2}
	Under  the same assumptions of Lemma \ref{lem:1}, if $[\tilde A~ \tilde \bsb  ]= [A~ \bsb]+\big[\sum_{i=1}^{t} \Delta a_i S_i ~\,\Delta \bsb\big]$ with $\big\|[\Delta \bsa^\top~\Delta \bsb^\top]^\top\big\|_2$ sufficiently small, then, for the STTLS solution $\bsx_k$ of $A\bsx \approx \bsb$ and the STTLS solution $\tilde \bsx_{k}$ of $\tilde A \bsx \approx \tilde \bsb$, we have
	$$
	\tilde \bsx_k =\bsx_k + M_k \begin{bmatrix}
		\ccM & 0\\
		0& I_m
	\end{bmatrix} \begin{bmatrix}
		\Delta \bsa\\
		\Delta \bsb
	\end{bmatrix}+ \Oh\left(\left\| \begin{bmatrix} \Delta \bsa\\  \Delta \bsb \end{bmatrix} \right\|_2^2 \right),
	$$
	where $\ccM  =\left[ \vect (S_1), \ldots, \vect(S_t)\right ] \in \R^{mn \times t}$.
	\end{lemma}
	
	
 The following theorem concerns with the explicit expressions for  the structured normwise,  mixed and componentwise condition numbers $\kappa_s(\bsa,\bsb)$, $m_s(\bsa,\bsb)$,  and $c_s(\bsa,\bsb)$ defined in Definition {\rm \ref{def2}} when $A$ can be expressed by \eqref{eq:A linear}. Since the proof is similar to  Theorem \ref{th:mixed}, we omit it here.

\begin{theorem}\label{th:mixed s}
Suppose the truncation level $k$ is chosen such that $V_{22} \neq 0$ and $\sigma_{k}> \sigma_{k+1}$. The absolute structured normwise, mixed and componentwise condition numbers $\kappa_s(\bsa,\bsb)$, $m_s(\bsa,\bsb)$, and $c_s(\bsa,\bsb)$ defined in Definition {\rm \ref{def2}} for the STTLS problem \eqref{TLS:definition} and \eqref{eq:A linear} can be characterized by
\[
		\kappa_s(\bsa,\bsb)= \left\| M_k \begin{bmatrix}
		\ccM & 0\\
		0& I_m
	\end{bmatrix} \right \|_2,
			m_s(\bsa,\bsb)= \frac{\left\| ~\left|M_k \begin{bmatrix}
		\ccM & 0\\
		0& I_m
	\end{bmatrix} \right| \begin{bmatrix}
			|\bsa| \\
			|\bsb|
		\end{bmatrix}  \right\|_\infty   }{ \|\bsx_k \|_\infty},
			c_s(\bsa,\bsb)=\left\|  \frac{~\left|M_k \begin{bmatrix}
		\ccM & 0\\
		0& I_m
	\end{bmatrix} \right|  \begin{bmatrix}
			|\bsa|\\
			|\bsb|
		\end{bmatrix}  }{\bsx_k } \right \|_\infty.
\]
	\end{theorem}

\begin{remark}\label{remark:2}
		Based on Definition \ref{def2} and Theorem \ref{th:mixed s}, the relative normwise condition number for the STTLS problem  \eqref{TLS:definition} and   \eqref{eq:A linear}  can be defined and has the following expression
		\begin{align}\label{eq:k-srel}
			\kappa_{s}^{\rm rel}  (\bsa,\bsb)&=\lim_{\epsilon \to 0} \sup_{\left\| \begin{bmatrix} \Delta \bsa\\  \Delta \bsb \end{bmatrix} \right\|_2 \leq~\epsilon~\left\| \begin{bmatrix}  \bsa\\  \bsb \end{bmatrix} \right\|_2} \frac{\left\| \psi_{s,k} ((\bsa,\, \bsb)+(\Delta \bsa,\, \Delta\bsb) )-\psi_{s,k} (\bsa,\,\bsb) \right\|_2   }{ \epsilon \|\bsx_k\|_2 } \nonumber\\
&= \frac{\left\| M_k \begin{bmatrix}
		\ccM & 0\\
		0& I_m
	\end{bmatrix} \right \|_2 \left\| \begin{bmatrix}  \bsa\\ \bsb \end{bmatrix} \right\|_2}{\|\bsx_k\|_2}.
		\end{align}
	Similar to \eqref {eq:realtion}, we have
		\begin{equation}\label{eq:relation2}
		m_s(\bsa,\bsb) \leq \sqrt{(t+m)n} ~ \kappa_s^{\rm rel}  (\bsa,\bsb).
		\end{equation}
		In Example \ref{example:toeplitzk6}, we can see $m_s(A,\bsb)$ can be much smaller than $\kappa_s^{\rm rel}  (A,\bsb)$. Hence, structured condition number can explain that structure-preserving algorithms  can enhance the accuracy of the numerical solution, since  structure-preserving algorithms preserve the underlying matrix structure. 
	\end{remark}

In the following proposition, we show that, when $A$ is a linear structured matrix defined by \eqref{eq:A linear}, the  structured normwise,  mixed and componentwise condition numbers $\kappa_s(\bsa,\bsb)$, $m_s(\bsa,\bsb)$,  and $c_s(\bsa,\bsb)$ are smaller than the corresponding unstructured condition numbers $\kappa(A,\bsb)$, $m(A,\bsb)$ and $c(A,\bsb)$ respectively.

\begin{proposition}\label{pro:relation of the unstru and the stru}
Using the notations above, we have $\kappa_s(\bsa,\bsb) \leq \kappa(A,\bsb)$. Moreover,  if $|A|=\sum_{i=1}^t |a_i| |S_i|$, then we have
\begin{align*}
	m_s(\bsa,\bsb)\leq m(A,\bsb),\quad c_s(\bsa,\bsb)\leq c(A,\bsb).
\end{align*}	
\end{proposition}

\begin{proof}
From \cite[Theorem 4.1]{LiJia2011}, the matrix $\ccM$ is column orthogonal. Hence, $\|\ccM\|_2=1$ and it is not difficult to see that
$\kappa_s(\bsa,\bsb) \leq \kappa(A,\bsb)$ by comparing their expressions. Using the monotonicity of infinity norm, it can be obtained that
	\begin{align*}
		\left\| ~\left|M_k \begin{bmatrix}
		\ccM & 0\\
		0& I_m
	\end{bmatrix} \right| \begin{bmatrix}
			|a| \\
			|b|
		\end{bmatrix}  \right\|_\infty \leq \left\| ~| M_k | ~\begin{bmatrix}
		|\ccM| & 0\\
		0& I_m
	\end{bmatrix}  \begin{bmatrix}
			|\bsa| \\
			|\bsb|
		\end{bmatrix}  \right\|_\infty = \left\| ~| M_k |   \begin{bmatrix}
			\vect(|A|)  \\
			|\bsb|
		\end{bmatrix}  \right\|_\infty,
	\end{align*}
	therefore we prove that $m_s(\bsa,\bsb)\leq m(A,\bsb)$ and $c_s(\bsa,\bsb)\leq c(A,\bsb)$ can be proved similarly. \qed
\end{proof}



\section{\bf Revisiting condition numbers of the untruncated TLS problem}\label{sect:compar}


In this section,  we investigate the relationship between  normwise, componentwise and mixed condition numbers for the TTLS problem and the previous corresponding  counterparts for the untruncated TLS.  In the following, let $\bsx_n$ be the untruncated TLS solution to \eqref{tls}. First let us review previous results on condition numbers for the untruncated TLS  problem. 

Let $\widetilde{\sigma}_n$ be the smallest singular value of $A$. As noted in  \cite{GolubTLS1980},  if
\begin{equation}\label{eq:genericity}
\widetilde{\sigma}_n>\sigma_{n+1},
\end{equation}
then the TLS problem \eqref{tls} has  a unique TLS solution
\[
\bsx_n = (A^\top A- \sigma_{n+1}^2I_n)^{-1}A^\top \bsb.
\]

Let $L^\top \bsx_n$ be a linear function of the TLS solution $\bsx_n$, where $L\in \R^{n \times l}$ is a fixed matrix with $l \leq n$.
We define the mapping
\begin{equation}\label{eq:h}
	h \quad : \quad \R^{m\times n} \times \R^m \rightarrow \R^l\quad : \quad (A,\, \bsb)  \mapsto L^\top \bsx_n=L^\top(A^\top A- \sigma_{n+1}^2I_n)^{-1}A^\top \bsb.
\end{equation}
As in \cite{Baboulin2011SIMAX}, the {\em absolute} normwise condition number of $L^\top \bsx$ can be characterized by
\begin{align}\label{eq:norm}
	\kappa_1 (L, A,\bsb)&=\max_{[\Delta A,~ \Delta \bsb]\neq 0} \frac{\|h'(A,\bsb) \cdot (\Delta A,\Delta \bsb) \|_2}{\big\|[\Delta A~\Delta \bsb] \big\|_F} \nonumber\\
&= \left ( 1+\|\bsx_n\|_2^2 \right ) ^{1/2}  \left\| L^\top  P^{-1} \Big( A^\top A +  \sigma_{n+1}^2 \bigg( I_n - \frac{2\bsx_n \bsx_n^\top}{1+\|\bsx_n \|_2^2} \Big) \Big)P^{-1} L \right\|^{1/2}_2,
\end{align}
where
\begin{equation}\label{eq:p}
	P=A^\top A-\sigma_{n+1}^2 I_n.
\end{equation}
Later, in \cite{JiaLi2013}, an equivalent expression of  $\kappa_1(I_n, A,\bsb)$  was given by
\begin{equation}\label{eq:k1}
	\kappa_1(I_n, A,\bsb)=\sqrt{1+\|\bsx_n\|_2^2}\left\|V_{11}^{-\top}S\right\|^{1/2}_2,
\end{equation}
where $V_{11}$ is defined by \eqref{eq:vp} with $k=n$ and $S={\rm diag}([s_1,\ldots,s_n]^\top)$ with
\[
s_i=\frac{\sqrt{\sigma_{i}^2+\sigma_{n+1}^2}}{\sigma_i^2-\sigma_{n+1}^2}.
\]

Recall that $\kappa(A,\bsb)$ is given by \eqref{eq:k ex}.  The relationship between the upper bound for $\kappa(A,\bsb)$ and the corresponding counterpart for $\kappa_1(I_n, A,\bsb)$ was studied in \cite[\S 2.5]{gtt13max}.  The following theorem shows the equivalence of $\kappa(A,\bsb)$ and $\kappa_1(I_n, A,\bsb)$.
\begin{theorem}
For the untruncated  TLS problem \eqref{tls}, the explicit expression of $\kappa(A,\bsb)$ given by \eqref{eq:k ex} with $k=n$ is  equivalent to that of $\kappa_1(I_n, A,\bsb)$ given by \eqref{eq:k1}
\end{theorem}
\begin{proof}
When $k=n$, it is easy to see that $\Pi_{n,1}=\Pi_{1,n}=I_n$. Also we have
\begin{align}\label{eq:Mn}
	M_n&=\frac{1}{V_{22}^2} [I_n\quad \bsx_n] V K D_*^{-1} [I_n \otimes \Sigma_2^\top\quad \Sigma_1] W,
\end{align}
where $\Sigma_1=\diag([\sigma_1,\,\dots\,, \sigma_n]^\top) \in \R^{n \times n }$, $\Sigma_2=\left[\sigma_{n+1},0,\ldots, 0 \right]^\top \in \R^{m-n}$, $D_*=\Sigma_1^2-\sigma_{n+1}^2 I_n$, and
\BE\label{eq:w}
K=\begin{bmatrix}
		V_{22} I_n\\
		V_{21}
	\end{bmatrix},\quad
W=\begin{bmatrix}
	V_1^\top \otimes U_2^\top \\[2mm]
	V_2^\top \otimes U_1^\top
\end{bmatrix}.
\EE

When $k=n$, under the genericity condition \eqref{eq:genericity},  the following identities hold for the TLS solution $\bsx_n$ (cf. \cite{GolubTLS1980})
\begin{align}\label{ex:x}
 \begin{bmatrix}\bsx_n\cr -1\end{bmatrix} = -\frac{1}{V_{22}}V_{2}= -\frac{1}{V_{22}} \begin{bmatrix}
 	V_{12} \\ V_{22}
 \end{bmatrix} ,\quad   V_{22}=\frac{1}{\sqrt{1+\bsx_n^\top \bsx_n }},
\end{align}
where $V$ has the partition in \eqref{eq:vp} and $\bsx_n=-V_{12}/V_{22}$ given by \eqref{eq:xk}.
Thus it is not difficult to see that
\begin{align}\label{eq:th2.1a}
	 \frac{1}{V_{22}^2}  [I_n\quad \bsx_n] V K&=\frac{1}{V_{22}^2}  [I_n\quad  \bsx_n] \begin{bmatrix}
	 	V_{11} & V_{12} \\
	 	V_{21} & V_{22}
	 \end{bmatrix}  \begin{bmatrix}
		V_{22} I_n\\
		V_{21}
	\end{bmatrix}=\frac{1}{V_{22}} \left(   V_{11}-\frac{1}{V_{22}}  V_{12 }V_{21} \right).\end{align}

Since
$$
\begin{bmatrix}
  V_{11}^\top &V_{21}^\top \cr V_{12}^\top &V_{22}
\end{bmatrix}
\begin{bmatrix}
  V_{11}&V_{12}\cr V_{21} &V_{22}
\end{bmatrix}
=
\begin{bmatrix}
  I_{n}&0\cr 0& 1
\end{bmatrix},
$$
we know that
$$
V_{11}^\top V_{11}+V_{21}^\top  V_{21} =I_n,\quad V_{11}^\top V_{12}+V_{22}V_{21}^\top =0,
$$
thus, it can be verified that
\begin{equation}\label{eq:th2.1b}
	I_n
 =V_{11}^\top \left(V_{11}-\frac{1}{V_{22}}V_{12} V_{21} \right).
\end{equation}
Combining \eqref{eq:th2.1a} and \eqref{eq:th2.1b} with the expression of $M_n$ given by \eqref{eq:Mn}, we have
\begin{equation}\label{eq:MnN}
	M_n=\frac{1}{V_{22} } V_{11}^{-\top } D_*^{-1}  [I_n \otimes \Sigma_2^\top\quad \Sigma_1] W.
\end{equation}
This, together with $WW^\top = I_{mn}$, yields
\begin{align*}
	M_n M_n^\top
	= (1+\|\bsx_n\|_2^2 ) V_{11}^{-\top } S^2 V_{11}^{-1}  ,
\end{align*}
where $S$ is defined in \eqref{eq:k1}. Therefore, when $k=n$ the expression of  $\kappa(A,\bsb)$ given by \eqref{eq:k ex} is reduced to
$$
\kappa(A,\bsb)= \left \| M_n \right\|_2=\|M_nM_n^\top \|_2^{1/2}= \kappa_1(I_n, A,\bsb).
$$

The proof is complete. 
\qed 	
\end{proof}
\

In \cite{Zhou}, Zhou et al. defined and derived the relative mixed and componentwise condition numbers  for the untruncated TLS problem \eqref{tls} as follows: Let  $[\tilde A~ \tilde \bsb]=[A~\bsb]+[\Delta A~\Delta \bsb]$, where $\Delta A$ and $\Delta \bsb$ are the perturbations of $A$ and $\bsb$ respectively. When the norm $\|[\Delta A,\Delta \bsb]\|_F$ is small enough, for the TLS solution $\bsx_n$ of $A\bsx \approx \bsb$ and the TTLS solution $\tilde \bsx_{n}$ of $\tilde A \bsx \approx \tilde \bsb$, we have
\begin{eqnarray}
  m_1(A,\bsb)=\lim_{\epsilon \to
0}\sup_{\substack{|\Delta A| \le\epsilon  |A|,\atop |\Delta \bsb|\leq \epsilon|\bsb| }}\frac{\|\tilde \bsx_n-\bsx_n\|_\infty}{\epsilon \|\bsx_n\|_\infty}=\frac{\Big\|\big|M+N\big| ~\vect\Big(\big[\,|A|~|\bsb|\big]\Big)\Big\|_\infty}{\|\bsx\|_\infty},\label{eq:mixed}\\
  c_1(A,\bsb)=\lim_{\epsilon \to
0}\sup_{\substack{|\Delta A| \le\epsilon  |A|,\atop |\Delta \bsb|\leq \epsilon |\bsb| }} \frac{1}{\epsilon} \left\|\frac{\tilde \bsx_n-\bsx_n}{\bsx_n}\right\|_\infty =\left\|\frac{\left|M+N\right| ~\vect\Big(\big[|A|~|\bsb|\big]\Big)  } {\bsx_n}\right\|_\infty,\label{eq:mixed2}
\end{eqnarray}
where
\begin{align*}
M & = \begin{bmatrix}P^{-1}\otimes \bsb^\top-\bsx_n^\top \otimes(P^{-1}A^\top)  &\quad P^{-1}A^\top \end{bmatrix},\,
 N  = 2\sigma_{n+1}P^{-1}\bsx_n(\bsv_{n+1}^\top \otimes \bsu_{n+1}^\top),
\end{align*}
and $\bsv_{n+1}$ are the $(n+1)$-th column of $U$ and $V$ respectively.

Recently,  in \cite{ds18}, Diao and Sun defined and gave mixed and componentwise condition numbers for the linear function $L^\top \bsx_n$ as follows.
 \begin{eqnarray*}
m_{1,L}(A,\bsb) &=&\frac{\left\|~ \left| L^\top P^{-1}\left(\bsx_n^\top \otimes \left(A^\top+\frac{2\bsx_n r^\top}{1+\bsx_n^\top \bsx_n} \right )-I_n\otimes \bsr^\top \right)\right| \vect (|A|)
+\left|L^\top P^{-1}\left(A^\top+\frac{2\bsx_n \bsr^\top}{1+\bsx_n^\top \bsx_n}\right)\right| |\bsb|\right\|_\infty} {\| L^\top \bsx_n\|_\infty},\label{eq:mcp} \\
c_{1,L}(A,\bsb)&=&\left\| D_{L^\top \bsx_n}^\dagger \left|L^\top  P^{-1}\left(\bsx_n^\top \otimes \left(A^\top+\frac{2\bsx_n \bsr^\top}{1+\bsx_n^\top \bsx_n} \right )-I_n\otimes \bsr^\top \right)\right| \vect (|A|) \right.  \nonumber \\
&& \left. \quad\quad \quad+D_{L^\top \bsx}^\dagger  \left|L^\top P^{-1}\left(A^\top+\frac{2\bsx_n \bsr^\top}{1+\bsx_n^\top \bsx_n}\right)\right| |\bsb|\right\|_\infty,\label{eq:mcp2}
\end{eqnarray*}
where $\bsr=\bsb-A\bsx_n$. Moreover, when $L=I_n$, the expressions of $m_{1,I_n}(A,\bsb)$ and $c_{1,I_n}(A,\bsb)$ were  equivalent to the explicit expressions of $m_1(A,\bsb)$ and $c_1(A,\bsb)$ given by \eqref{eq:mixed}--\eqref{eq:mixed2} (cf.\cite[Theorem 3.2]{ds18}).

 In the following theorem, we prove that, when $k=n$, $m(A,\bsb)$ and $c(A,\bsb)$ given by Theorem \ref{th:mixed} are reduced to those of $m_1(A,\bsb)$ and $c_1(A,\bsb)$, respectively.

 \begin{theorem}
 	Using the notations above, when $k=n$, we have $m(A,\bsb)=m_1(A,\bsb)$ and $c(A,\bsb)=c_1(A,\bsb)$.

 \end{theorem}

 \begin{proof} From the proof of \cite[Lemma 2]{DiaoWeiXie}, for $P$ given by \eqref{eq:p} we have 
 $$
 P=V_{11}D_*V_{11}^\top,
 $$
 where $D_*$ and $V_{11}$ are defined in \eqref{eq:Mn} and \eqref{eq:vp}, respectively. Thus, when $k=n$, using  \eqref{eq:MnN} and  \eqref{eq:w} we have
 \begin{align}\label{eq:MnN1}
 	M_n&=\frac{1}{V_{22} } V_{11}^{-\top } D_*^{-1} \left( V_1^\top \otimes (\Sigma_2^\top U_2^\top ) +\Sigma_1 (V_2^\top \otimes  U_1^\top )\right) \nonumber  \\
	&=\frac{1}{V_{22} } P^{-1} V_{11} \left( V_1^\top \otimes (\Sigma_2^\top U_2^\top ) +\Sigma_1 (V_2^\top \otimes  U_1^\top )\right)\nonumber\\
	&=P^{-1}\left(M_{n,1}+M_{n,2} \right),
\end{align}
where
\begin{align}\label{eq:Mi}
M_{n,2}= \frac{1}{V_{22} }  V_{11}  \Sigma_1 (V_2^\top \otimes  U_1^\top ),\, M_{n,1}= \frac{1}{V_{22} }  V_{11} \left ( V_1^\top \otimes (\Sigma_2^\top U_2^\top )  \right).
 \end{align}
Partition  $M_{n,1}$ and $M_{n,2}$ as follows:
 \begin{equation}\label{eq:mn12}
	M_{n,1}=[\ccN_{1}\quad\ccN_{2}], \quad M_{n,2}=[\ccT_{1}\quad\ccT_{2}],\quad  \ccN_{1},\ccT_{1}\in\R^{n\times mn}.
\end{equation}
By comparing the expressions of $m_{1,I_n }(A,\bsb)$ and $m(A,\bsb)$ with $k=n$, we only need to show that
\begin{align}\label{qnt12}
\ccN_1+\ccT_1=Q_1,\quad \quad \ccN_2+\ccT_2=Q_2,
\end{align}
where
\begin{align*}
Q_1=I_n\otimes \bsr^\top-\bsx_n^\top \otimes Q_2,\quad Q_2=A^\top+\frac{2\bsx_n \bsr^\top}{1+\bsx_n^\top \bsx_n}.
\end{align*}

Using  the SVD of $[A,\, \bsb]$ in \eqref{eq:svd}, the partitions of $V$, $U$ in \eqref{eq:vp} and $\Sigma$ in \eqref{eq:k ex}, it follows that
\begin{align}\label{eq:aform}
	A^\top
=\left ([A\quad \bsb] \begin{bmatrix}
		I_n \\ 0
	\end{bmatrix} \right)^\top
	= V_{11} \Sigma_1 U_1^\top + V_{12} \Sigma_2^\top U_{2}^\top = V_{11} \Sigma_1 U_1^\top + \sigma_{n+1} V_{12} \bsu_{n+1}^\top ,
	\end{align}
where $\bsu_{n+1}$ is the $(n+1)$-th column of $U$.  We note that
\begin{align}\label{eq:sigma2u}
	\Sigma_{2}=\sigma_{n+1} \bse_1^{(m-n)},\quad \Sigma_2^\top U_2^\top =\bsu_{n+1}^\top,
\end{align}
where  $\bse_1^{(m-n)}$ is the first column of $I_{m-n}$. From the SVD of $[A~\bsb]$ and \eqref{ex:x}, it is easy to check that
\begin{equation}\label{eq:r}
  \bsr=\bsb-A\bsx_n=-[A~\bsb]\begin{bmatrix}
    \bsx_n \cr -1
  \end{bmatrix}=\frac{1}{V_{22}}[A~\bsb]V_{2}=\frac{\sigma_{n+1}}{V_{22}}\bsu_{n+1}.
\end{equation}
Substituting \eqref{eq:aform}, \eqref{ex:x} and \eqref{eq:r} into the expression of $Q_2$ and $Q_1$ yields
\begin{align}
	Q_2&= V_{11} \Sigma_1 U_1^\top - \sigma_{n+1} V_{12} \bsu_{n+1}^\top, \label{eq:q2}\\
	Q_1&=\frac{\sigma_{n+1} }{V_{22}}I_n\otimes \bsu_{n+1}^\top+\frac{1 }{V_{22}} V_{12}^\top \otimes (V_{11} \Sigma_1 U_1^\top ) - \frac{\sigma_{n+1} }{V_{22}} V_{12}^\top \otimes  \left(  V_{12} \bsu_{n+1}^\top \right).\label{eq:q1}
\end{align}


Since
$
V_{11} V_{11}^\top +V_{12} V_{12}^\top =I_n,\,  V_{11} V_{21}^\top +V_{22} V_{12}=0,
$
we deduce that
\begin{align}\label{eq:v11v12}
	V_{11} V_{11}^\top =I_n- V_{12} V_{12}^\top ,\quad V_{11} V_{21}^\top =- V_{22} V_{12}.
\end{align}
Using \eqref{eq:Mi}, \eqref{eq:sigma2u}, and \eqref{eq:v11v12} we have
\begin{align}\label{eq:n12}
	M_{n,1}&= \frac{1}{V_{22} } \left ( V_{11}   V_1^\top \otimes (\Sigma_2^\top U_2^\top )\right)
=\frac{1}{V_{22} } \sigma_{n+1} \left[V_{11} V_{11}^\top\quad V_{11} V_{21}^\top \right]\otimes \bsu_{n+1}^\top \nonumber\\
&=\sigma_{n+1} \left[\frac{1}{V_{22}}(I_n- V_{12} V_{12}^\top) \otimes \bsu_{n+1}^\top \quad -V_{12} \otimes \bsu_{n+1}^\top \right],\nonumber\\
\ccN_{1}&=\frac{\sigma_{n+1}}{V_{22}}(I_n- V_{12} V_{12}^\top) \otimes \bsu_{n+1}^\top,\quad   \ccN_{2}=-\sigma_{n+1} V_{12} \otimes \bsu_{n+1}^\top = -\sigma_{n+1} V_{12} \bsu_{n+1}^\top.
\end{align}
Using the partition of $V_2^\top =[V_{12}^\top\quad V_{22}] $ and \eqref{eq:Mi} we have
\begin{align}\label{eq:t12}
   M_{n,2} &= \frac{1}{V_{22}}\left[V_{11}\Sigma_{1}(V_{12}^\top\otimes U_{1}^\top)\quad V_{11}\Sigma_{1}(V_{22}^\top\otimes U_{1}^\top)\right],\nonumber\\
\ccT_1&= \frac{1}{V_{22} }  V_{11}  \Sigma_1 (V_{12}^\top \otimes  U_1^\top )= \frac{1}{V_{22}}  V_{12}^\top \otimes (  V_{11}  \Sigma_1 U_1^\top  ),
\ccT_2=\frac{1}{V_{22}} V_{11} \Sigma_1 V_{22}  U_1^\top=  V_{11} \Sigma_1 U_1^\top.
\end{align}
From \eqref{eq:n12}, \eqref{eq:t12},  \eqref{eq:q1}, and \eqref{eq:q2}, it is easy to check that the two inequalities in \eqref{qnt12} hold.
%
The proof is complete.
%
%
 	\qed
 \end{proof}

\begin{remark}\label{re:3.1}
From \eqref{eq:MnN1}, \eqref{eq:mn12}, \eqref{qnt12}  we have
$$
M_n=P^{-1} \left[-\bsx_n^\top \otimes \left(A^\top+\frac{2\bsx_n \bsr^\top}{1+\bsx_n ^\top \bsx_n } \right )+I_n\otimes \bsr^\top  \quad A^\top+\frac{2\bsx_n\bsr^\top}{1+\bsx_n^\top \bsx_n} \right].
$$
\end{remark}

 The structured condition numbers for the untruncated TLS problem with linear structures were studied by
Li and Jia in \cite{LiJia2011}.  For the structured matrix $A$ defined by \eqref{eq:A linear}, denote
\begin{align}
\ccK &=P^{-1}
\left(2 A^\top \frac{\bsr\bsr^\top }{\|\bsr\|_2^2} G(\bsx_n )-A^\top G(\bsx_n ) + \left[
	I_n \otimes \bsr^\top \quad 0
\right] \right ),\quad  G(\bsx_n)=\left[
	\bsx_n^\top\quad -1
\right] \otimes I_m, \label{eq:K}
\end{align}
where  $P$ is defined by \eqref{eq:p}. The structured mixed condition number $m_{s,n}(\bsa,\bsb)$ is characterized
as \cite{LiJia2011}
\begin{eqnarray}\label{eq:str mixed}
m_{s,n} (\bsa,\bsb) &=& \lim_{\epsilon \rightarrow
0}\sup_{\substack{ |\Delta \bsa| \le\epsilon  |\bsa|,\\ |\Delta \bsb| \le\epsilon  |\bsb|  } }\frac{\| \tilde \bsx_n-\bsx_n\|_\infty }{\epsilon \|\bsx_n \|_\infty }  = \frac{\left\| ~\left |\ccK  \begin{bmatrix}
  \ccM & 0\cr
 0  &I_m
\end{bmatrix} \right| \begin{bmatrix}
	|\bsa| \\ |\bsb|
\end{bmatrix}\,\right\|_\infty  }{\|\bsx_n\|_\infty }.
\end{eqnarray}
In \cite[Theorem 4.3]{LiJia2011},  Li and Jia  proved that $m_{s,n} (\bsa,\bsb) \leq  m_1 (A,\bsb)$ and $\ccK=M_n$, where $m_1 (A,\bsb)$ is given by \eqref{eq:mixed}. Hence we have the following proposition. Indeed, we prove that, when $k=n$, the expression of $m_{s,n}(\bsa,\bsb)$  given by \eqref{eq:str mixed} is  reduced to  that of $m_s(\bsa,\bsb)$ in Theorem \ref{th:mixed s}.

\begin{proposition}\label{pro31}
Using the notations above, when $k=n$, we have $m_s(\bsa,\bsb)=m_{s,n}(\bsa,\bsb)$.
\end{proposition}


 \section{\bf Small sample statistical condition estimation}\label{sec:SCE}


Based on small sample statistical condition estimation (SCE), reliable condition estimation algorithms for both unstructured and structured normwise, mixed and componentwise are devised,  which utilize  the SVD of the augmented matrix $[A~\bsb]$.  In the following, we first review  the basic idea of SCE. Let   $\psi:\R^{p}\rightarrow \R$ be a differentiable function.  For the input vector $u$, we want to estimate the sensitivity of the output $\psi(\bsu)$ with respect to small perturbation $\epsilon \bsd$ on $\bsu$, where $\bsd$ is a unit vector and $\epsilon$ is a small positive number. The Taylor expansion of $\psi$ at $\bsu$ is given by
$$
\psi (\bsu+\epsilon \bsd)=\psi(\bsu)+\epsilon (\nabla \psi(\bsu))^\rt \bsd+\Oh(\epsilon ^2),
$$
where $\nabla \psi(\bsu)\in \R^p$ is the gradient of $\psi$ at
$\bsu$. Neglecting the second and higher order terms of $\epsilon$ we have
$$
\left| \psi (\bsu+\epsilon \bsd)- \psi (\bsu)\right|\approx \epsilon (\nabla \psi(\bsu))^\rt \bsd,
$$
from which we conclude that the local sensitivity can be measured by
$\|\nabla \psi(\bsu)\|_2$. Let the Wallis factor be given by  \cite{KenneyLaub_SISC94}
$$
\omega_p=
\begin{cases}
1, & \text{for}~p\equiv 1,\\
\frac{2}{\pi}, &\text{for}~p\equiv 2,\\
\frac{1 \cdot 3 \cdot 5 \cdots (p-2)}{2 \cdot 4 \cdot 6 \cdots (p-1)}, &\text{for}~p ~ \text{odd} ~\text{and} ~p>2, \\
\frac{2}{\pi} \frac{2 \cdot 4 \cdot 6 \cdots (p-2)}{1 \cdot 3 \cdot
5 \cdots (p-1)}, &\text{for}~ p ~ \text{even}~ \text{and} ~p>2. 
\end{cases}
$$
As in \cite{KenneyLaub_SISC94}, if $\bsd$ is selected  uniformly and randomly from the unit $p$-sphere
$B_{p-1}$ (denoted as $ \bsd \in \mathcal {U}(B_{p-1})$), then the expected value $
{\mathbb   E}(|(\nabla \psi (\bsu))^\rt \bsd|/\omega_p)$  is $\|\nabla
\psi (\bsu)\|_2$.  In practice, the Wallis factor can be
approximated accurately by \cite{KenneyLaub_SISC94}
\begin{equation}\label{ss:p} \omega_p \approx
\sqrt{\frac{2}{\pi(p-\frac{1}{2})}}.
\end{equation}
Therefore,  the following quantity
$$
\nu=\frac{\left|(\nabla \psi(\bsu))^\rt \bsd \right|}{\omega_p}
$$
can be used as a condition estimator for  $\|\nabla \psi (\bsu)\|_2$
with high probability for the function $\psi$ at $\bsu$. For example, let  $\gamma >1$, which indicates the accuracy of the estimator, it is shown that
$$
{\mathbb P}\left(\frac{\|\nabla \psi (\bsu)\|_2}{\gamma}\leq \nu \leq \gamma
\|\nabla \psi(\bsu)\|_2\right) \geq 1-\frac{2}{\pi
\gamma}+\Oh\left(\frac{1}{\gamma^2}\right).
$$
In general, we are interested in
finding an estimate that is accurate to a factor of 10 ($\gamma =10$). The accuracy of the condition estimator can enhanced by using multiple samples of $\bsd$, denoted $\bsd_j$. The $\ell $-sample condition estimation is given by
$$
\nu(\ell )=\frac{\omega_\ell }{\omega_p }\sqrt{\sum_{j=1}^\ell \left|\nabla \psi(\bsu)^\rt \bsd_j\right|^2
},
$$
where the matrix $[\bsd_1,\ldots,\bsd_\ell ]$ is orthonormalized after $\bsd_1,\ldots,\bsd_\ell $ are selected uniformly and
randomly from $\mathcal {U}(B_{p-1})$.  Usually, at most two or three samples are sufficient for high accuracy. For example, the accuracies of $\nu(2)$ and $\nu(3)$ are given by \cite{KenneyLaub_SISC94}
\begin{align*}
{\mathbb P}\left(\frac{\|\nabla \psi (\bsu)\|_2}{\gamma}\leq \nu(2) \leq \gamma
\|\nabla \psi(\bsu)\|_2\right) &\approx 1-\frac{\pi}{4 \gamma^2},\quad \gamma>1,\\
{\mathbb P} \left(\frac{\|\nabla \psi(\bsu)\|_2}{\gamma}\leq \nu(3) \leq \gamma
\|\nabla \psi(\bsu)\|_2\right) &\approx 1-\frac{32 }{3 \pi^2 \gamma^3},\quad \gamma>1.
\end{align*}
As an illustration, for $\ell =3$ and $\gamma=10$, the estimator $\nu(3)$ has probability $0.9989$, which is within a relative factor $10$ of the true condition number $\|\nabla \psi(\bsu)\|_2$.

The above results can be easily extended  to vector-valued or matrix-valued functions.
\subsection{Normwise perturbation analysis}\label{subsec:norm}

In this subsection,  we propose an algorithm  for the normwise condition
estimation of the TTLS problem (\ref{TLS:definition}) based on SCE.  The input data of Algorithm \ref{algo:subnorm} includes the matrix $A \in \R^{m \times n}$, the vector $\bsb \in \R^m$,  the SVD of $[A~ \bsb]$, and the computed solution $\bsx_k \in \R^n$. The output includes the condition vector $\mathscr{K}_{\rm abs}^{\mathrm{TTLS},(\ell)}$ and the estimated relative condition number $\kappa_{\rm SCE}^{{\rm TTLS},(\ell)}$.

\begin{algorithm}
\caption{Small sample condition estimation for the TTLS problem under normwise perturbation analysis}
\label{algo:subnorm}
\begin{itemize}
\item[1.] Generate matrices $[\Delta \hat{A}_1~\Delta \hat{\bsb}_1],\ldots, [\Delta \hat{A}_\ell~\Delta \hat{\bsb}_\ell ]\in\R^{m\times (n+1)}$ with entries being in ${\mathcal N}(0,1)$, the standard Gaussian distribution. Orthonormalize the following matrix
$$
\left[\begin{matrix}\vect(\Delta \hat{A}_1)&\vect(\Delta\hat{ A}_2)&\cdots&
\vect(\Delta \hat{A}_\ell )\cr \Delta \hat{\bsb}_1&\Delta \hat{\bsb}_2&\cdots&
\Delta \hat{\bsb}_\ell
\end{matrix}\right]
 $$
to obtain $[\bsq_1~\bsq_2,\ldots,\bsq_\ell ]$ via the modified Gram-Schmidt orthogonalization process.
Set
\[
\Delta H_i:=[\Delta A_i ~\Delta \bsb_i]={\sf unvec}(\bsq_i),\quad i=1,\ldots,\ell.
\]
\item[2.] Let $p=m(n+1)$. Approximate $\omega_p$ and $\omega_\ell $
by~(\ref{ss:p}).
\item[3.] For $i=1,2,\ldots,\ell $, compute
\begin{equation}\label{eq:al y}
\bsg_i= \frac{1}{\|V_{22}\|_2^2}\left(V_{11} \, (Z_1^\top + Z_2)V_{22}^\top +V_{12} \, (Z_1+Z_2^\top)V_{21}^\top  + \bsx_k\sum_{j=1}^4 c_j \right),
\end{equation}
where $Z_i$ and $c_i$ are defined in \eqref{eq:dir derivative x} with $\Delta H=\Delta H_i$. Estimate the absolute condition vector
\begin{eqnarray*}
\mathscr{K}_{\rm
abs}^{\mathrm{TTLS},(\ell )}&=&\frac{\omega_\ell }{\omega_p}\sqrt{\sum_{j=1}^\ell |\bsg_j|^2}.
\end{eqnarray*}
Here, for any vector $\bsg=[g_1,\ldots, g_n]^\top\in\Rn$, ${|\bsg|^2}=[|g_{1}|^2,\ldots,|g_{n}|^2]^\top$ and  $\sqrt{|\bsg|}=[\sqrt{|g_{1}|}, \ldots, \sqrt{|g_{n}|}]^\top$.
\item[4.] Compute the normwise condition number as follows,
$$
\kappa_{\rm SCE}^{{\rm TTLS},(\ell )}=\frac{N_{\rm
SCE}^{{\rm TTLS},(\ell )} \big\|[A~\bsb]\big\|_F}{||\bsx_k ||_2},
$$
where $
N_{\rm
SCE}^{{\rm TTLS},(\ell)}:=\frac{\omega_\ell }{\omega_p}\sqrt{ \sum_{j=1}^\ell ||\bsg_j||_2^2}=\|\mathscr{K}_{\rm
abs}^{\mathrm{TTLS},(\ell )}\|_F.$
\end{itemize}
\end{algorithm}

Next, we  give some remarks on the computational cost of Algorithm \ref{algo:subnorm}.
  In Step 1,  the modified Gram-Schmidt orthogonalization process  \cite{GolubVanLoan2013Book} is adopted  to form an orthonormal matrix $[\bsq_1~\bsq_2,\ldots,\bsq_\ell]$
  and the total flop count is about $\Oh(mn\ell^2)$.
 The cost associated with step 3 is about $\Oh(\ell k(mn+m(m-\ell)+(5k-n-1)(n+1-k)))$ flops that is mainly from computing the directional derivative in Lemma~\ref{pro1}. The last step needs $\Oh(mn+n\ell )$ flops.
We note that $\ell=3$  generates a good condition estimation. In this case, the total cost of Algorithm~\ref{algo:subnorm} is  $\Oh(mnk+m^2k+ k(5k-n)(n+1-k))$, which does not exceed the cost of computing the SVD of $[A~\bsb]$ and $\bsx_k$. Furthermore, the directional derivative \eqref{eq:al y} only be computed once in one loop of Algorithm~\ref{algo:subnorm}. On the contrary, Gratton et. al \cite{gtt13max} proposed the normwise condition number estimation algorithm  through using  the power method \cite{Higham2002Book}, which needs to evaluate the 
matrix-vector products  $M_k \bsf$ and $M_k^\top \bsg$ in one loop for some suitable dimensional  vectors $\bsf$ and $\bsg$. Therefore, Algorithm~\ref{algo:subnorm} is more efficient compared with the normwise condition number estimation method in \cite{gtt13max}.

\subsection{Componentwise perturbation analysis}\label{subsec:comp}

If the perturbation in the input data is measured componentwise rather
than by norm, it may help us to measure  the sensitivity of
a function more accurately \cite{Skeel79}. The SCE method  can also be used to measure the sensitivity  of componentwise perturbations \cite{KenneyLaub_SISC94}, which may give a more realistic indication
of the accuracy of a computed solution than that from the normwise condition number.  In componentwise perturbation analysis, for a perturbation $\Delta A=(\Delta a_{ij})$ of  $A=(a_{ij}) \in \R^{m \times n}$, we assume that $|\Delta A| \leq \epsilon |A|$. Therefore, the perturbation $\Delta A$ can be rewritten as  $\Delta A=\delta\, ( {\mathcal A} \boxdot A)$ with
$|\delta|\leq \epsilon $ and each entry of    $\mathcal A$  being in the interval $[-1,1]$.  Based on the above observations, we can obtain
a componentwise sensitivity estimate of the solution $x_k $ of the TTLS problem (\ref{TLS:definition}) as follows. The detailed descriptions are given in Algorithm \ref{eq:es mixed}, which is a modification of Algorithm \ref{algo:subnorm} directly.

\begin{algorithm} \caption{Small sample condition estimation for
the TTLS problem under componentwise perturbation analysis}\label{eq:es mixed}
\begin{itemize}
\item[1.] Generate matrices $[\Delta \hat{A}_1 ~\Delta \hat{\bsb}_1]~
[\Delta \hat{A}_2 ~\Delta \hat{\bsb}_2],\ldots,[\Delta \hat{A}_{\ell} ~\Delta \hat{\bsb}_{\ell}] \in \R^{m\times (n+1)}$ with entries being in ${\mathcal N}(0,1)$. Orthonormalize the following matrix
$$
\left[\begin{matrix}\vect(\Delta \hat{A}_1 )&\vect(\Delta \hat{A}_2)&\cdots&
\vect(\Delta \hat{A}_{\ell})\cr \Delta \hat{\bsb}_1&\Delta \hat{\bsb}_2&\cdots&
\Delta \hat{\bsb}_{\ell}
\end{matrix}\right]
 $$
to obtain $[\bsq_1~\bsq_2,\ldots,\bsq_\ell]$ via the modified Gram-Schmidt orthogonalization process.
Set
\[
[\Delta \widetilde {A}_i ~\Delta \tilde \bsb_i]={\sf unvec}(\bsq_i)\quad i=1,\ldots,\ell.
\]
Set
\[
\Delta H_i:=[\Delta A_i~\Delta \bsb_i]= [\hat{A_i}~\hat{\bsb_i}]\boxdot [\Delta \widetilde A_i~\Delta\tilde \bsb_i]\quad i=1,\ldots,\ell.
\]
\item[2.] Let $p=m(n+1)$. Approximate $\omega_p$ and $\omega_{\ell}$
by~(\ref{ss:p}).
\item[3.] For $j=1,2,\ldots,\ell,$ calculate $\bsg_j$ by (\ref{eq:al y}).
Estimate the absolute condition vector
\begin{eqnarray*}
C_{\rm abs}^{\mathrm{TTLS},(\ell )}=\frac{\omega_\ell }{\omega_p}\sqrt{\sum_{j=1}^\ell |\bsg_j|^2}.
\end{eqnarray*}
\item[4.]Set the relative condition vector
$C_{\rm rel}^{\mathrm{TTLS},(\ell )}=C_{\rm abs}^{\mathrm{TTLS},(\ell )}/\bsx_k $.
Compute the mixed and componentwise condition estimations
$m_{\rm SCE}^{\mathrm{TTLS},(\ell )}$ and $c_{\rm SCE}^{\mathrm{TTLS},(\ell)}$ as follows,
\begin{equation*}
m_{\rm SCE}^{\mathrm{TTLS},(\ell )} =
\frac{\left\|C_{\rm abs}^{\mathrm{TTLS},(\ell )}\right\|_{\infty}}{\|\bsx_k\|_{\infty}}, \quad
c_{\rm SCE}^{\mathrm{TTLS},(\ell )} =\left\|C_{\rm rel}^{\mathrm{TTLS},(\ell )}\right\|_\infty=
\left\|\frac{C_{\rm abs}^{\mathrm{TTLS},(\ell )}}{\bsx_k}\right\|_{\infty}.
\end{equation*}
\end{itemize}
\end{algorithm}


To estimate the mixed and componentwise condition numbers via Algorithm \ref{eq:es mixed}, we only need additional computational cost $\ell m(n+1)$ flops comparing with Algorithm~\ref{algo:subnorm}. 

\subsection{Structured perturbation analysis}\label{subsec:str}

For the structured TLS problem, it is reasonable to consider the case that the perturbation $\Delta A$ has the same structure as $A$.
Suppose the matrix $A$ takes the form of \eqref{eq:A linear}, i.e.,  $A = \sum_{i = 1}^ta_i S_i$. Here, $S_1,\ldots,S_t$ are linearly independent matrices of a linear subspace ${\mathcal S} \subset \R^{m\times n}$.
It is easy to see that
$$
\vect (A) = \Phi^{st} \bsa,
$$
where $\bsa= [a_1, a_2,\ldots, a_t]^\top \in \R^t$ and $\Phi^{st} = [\vect (S_1), \vect (S_2),\ldots, \vect (S_t)]$.
For
\[
\mathcal S=\{\mbox{all $m\times n$ real Toeplitz matrices}\},
\]
we have $t=m+n-1$,
\begin{align*}
	 S_{1}&={\tt toeplitz}(\bse_{1},0),\,\ldots, \,
  S_{m}={\tt toeplitz }(\bse_m,0 ),\\ S_{m+1}&={\tt  toeplitz}(0,\bse_2 ) \ldots, \,
  S_{m+n-1}={\tt  toeplitz }(0,\bse_n),
\end{align*}
where the \textsc{Matlab}-routine  notation $A={\tt toeplitz}(T_c,T_r)\in{\mathcal S}$ denotes a Toeplitz matrix having  $T_c\in\Rm$ as its first column and  $T_r\in\Rn$ as its first row, and $A= \sum_{i = 1}^ta_i S_i$, where $\bsa=[T_c^{\top}, T_r(2:end)]^{\top}\in \R^{m+n-1}$. This means that a Toeplitz matrix $A$ can be obtained by taking ${\mathcal S}=\R^{m+n-1}$ and letting $\tau$ be the map
$$
\tau(\bsa)=\begin{bmatrix}
  a_{1}     & a_{m+1}&\cdots & a_{m+n-2} &a_{m+n-1}\cr
  a_{2} & a_1    & \cdots  &                    a_{m+n-3} &a_{m+n-2} \cr
  \vdots &\vdots&\ddots & \vdots & \vdots\cr
  a_{m-1} &a_{m-2} &\cdots & a_{n-1} &a_{n}\cr
  a_{m} & a_{m-1} &\cdots & a_{n-2} &a_{n-1}
\end{bmatrix}=A \in \R^{m \times n},\quad \forall \bsa=\begin{bmatrix}
  a_1     \cr
  a_2 \cr
  \vdots \cr
  a_{m+n-1}
\end{bmatrix}\in\R^{m+n-1}.
$$

The SCE method maintains the desired matrix structure by working with the  perturbations of $A$ and $\bsb$ in the linear space of ${\mathcal S}\times \R^m$. This produces only slight changes in the SCE algorithm. By simply generating $\Delta a_i$ and $\Delta \bsb$ randomly instead of $\Delta A=$ and $\Delta \bsb$ as in Algorithm~\ref{algo:subnorm}, we obtain an algorithm to estimate the condition of a composite map $f\circ\tau$. We summarize the structured normwise condition estimation  in Algorithm \ref{algstr}, which also includes  the structured componentwise condition estimation. The computational cost of Algorithm \ref{algstr} is reported in Table \ref{tab:complexity_structuredsce}.

\begin{algorithm} \caption{Small sample condition estimation for the STTLS problem under structured perturbation analysis}
\label{algstr}
\begin{itemize}
\item[1.] Generate matrices $[\Delta \hat{\bsa}_1,\Delta \hat{\bsa}_2, \ldots,\Delta \hat{\bsa}_{\ell}],~ [\Delta \hat{\bsb}_1,\Delta \hat{\bsb}_2, \ldots, \Delta \hat{\bsb}_{\ell}]$ with
entries in ${\mathcal{N}}(0,1),$ where $\Delta \hat{\bsa}_i \in \R^t$ and $\Delta \hat{\bsb}_i \in \R^m$. Orthonormalize the following matrix
$$
\left[\begin{matrix}\Delta \hat{\bsa}_1 & \Delta \hat{\bsa}_2 &\cdots&
\Delta \hat{\bsa}_{\ell} \cr \Delta \hat{\bsb}_1& \Delta \hat{\bsb}_2&\cdots&
\Delta \hat{\bsb}_{\ell}
\end{matrix}\right],
$$
to obtain an orthonormal matrix $[\xi_1~\xi_2,\ldots,\xi_{\ell}]$ by using the modified Gram-Schmidt orthogonalization process.
Set
\[
[\Delta \tilde \bsa_i^\top ~ \Delta  \tilde \bsb_i^\top]^\top=\xi_i,\quad i=1,\ldots, \ell.
\]
Set
\[
[\Delta \bsa_{i}^{\top}~ \Delta \bsb_i^{\top}]^\top=[ \hat{\bsa}_i^{\top}~ \hat{\bsb}_i^{\top}]^\top\boxdot [\Delta  \tilde \bsa_i^\top ~ \Delta  \tilde \bsb_i^\top]^\top,\quad i=1,\ldots, \ell
\]
and
\[
\Delta H_i:=[\Delta A_i\quad\Delta \bsb_i],\quad \Delta A_i=\sum_{j=1}^t\Delta a_jS_j, \quad i=1,\ldots,\ell.
\]

\item[2.]  Let $p=t+m$. Approximate $\omega_p$ and $\omega_{\ell}$ by (\ref{ss:p}).

\item[3.] For $j=1,2,\ldots,\ell,$ calculate $\bsg_j$ by (\ref{eq:al y}).
Compute the absolute condition vector
\[
\bar{K}_{abs} = \frac{\omega_{\ell}}{\omega_t} \sqrt{ \sum_{j=1}^\ell |\bsg_j|^2}.
\]
\item[4.] Compute the normwise condition numbers as follows:
$$
\kappa_{\rm SCE}^{\mathrm{STTLS},(\ell)}=\frac{\left\|\bar{K}_{abs} \right\|_2 \left\|~[\bsa^\top~\bsb^\top]^\top\right\|_2}{\|\bsx_{k}\|_2},
$$
Compute the mixed and componentwise condition estimations
$m_{\rm SCE}^{\mathrm{STTLS},(\ell )}$ and $c_{\rm SCE}^{\mathrm{STTLS},(\ell)}$ as follows,
\begin{equation*}
m_{\rm SCE}^{\mathrm{STTLS},(\ell )} =
\frac{\left\|\bar{K}^{S}_{abs}\right\|_{\infty}}{\|\bsx_k\|_{\infty}}, \quad
c_{\rm SCE}^{\mathrm{TTLS},(\ell )} =
\left\|\frac{\bar{K}^{S}_{abs}}{\bsx_k}\right\|_{\infty}.
\end{equation*}
\end{itemize}
\end{algorithm}

\begin{table}[!htbp]
\caption{\label{tab:complexity_structuredsce} Computational complexity for Algorithm \ref{algstr}. }
\begin{center}
\begin{tabular}{ccccc}\hline
 Step & & 1 & 2 & 3  \\
\hline
  Algorithm \ref{algstr}  & & $\Oh(\ell^2(m+r)+\ell(m+r))$     &  $\Oh(mn)$      &      $\Oh(mn\ell+n^2\ell)$   \\
 \hline
\end{tabular}
\end{center}
\end{table}

\section{\bf Numerical examples}\label{sec:ex}


In this section, we  present  some numerical examples to illustrate the reliability of
the SCE for the TTLS problem \eqref{TLS:definition}. For a given TTLS problem, the TTLS solution $\bsx_k$ with truncation level $k$ can be computed by utilizing the SVD of $[A~\bsb]$ and \eqref{ex:x}.  The corresponding exact condition numbers are computed by their explicit expressions associated with the given data $[A~\bsb]$. All the sample number $\ell$ in Algorithms \ref{algo:subnorm} to \ref{algstr} are set to be $\ell=3$. 
All the numerical experiments are carried out on \textsc{Matlab} R2019b with the machine epsilon $\mu \approx 2.2 \times 10^{-16}$ under Microsoft Windows 10.

\begin{example}\label{example:small}
\em Let
$$A=\left[\begin{matrix} 2&0\\0&3\\0&10^{-s}\\
\end{matrix}\right]
,\quad
\bsb=\left[\begin{matrix}10^{-s}\\0\\1
\end{matrix}\right],
$$
where $s\in \mathbb R_+$.

\end{example}

In Table \ref{table:small}, we compare our SCE-based estimations $\kappa_{\rm SCE}^{{\rm TTLS},(\ell)}$, $m_{\rm SCE}^{\mathrm{TTLS},(\ell)}$ and $c_{\rm SCE}^{\mathrm{TTLS},(\ell)}$ from Algorithms \ref{algo:subnorm} and \ref{eq:es mixed} with the corresponding exact condition numbers for Example \ref{example:small}. The symbol  $\lq\lq\times"$ in Table \ref{table:small} means the condition numbers $\kappa^{rel}_1(A,\bsb)$, $m_1(A,\bsb)$ and $c_1(A,\bsb)$  are not defined for the truncation level $k=1$.  From the numerical results listed in Table \ref{table:small}, it is easy to find that the normwise condition number $\kappa^{rel}(A,\bsb)$ defined by \eqref{eq:k rel} are much greater than results of mixed and componentwise condition numbers $m(A,\bsb)$ and $c(A,\bsb)$ given by Theorem \ref{th:mixed}. The (untruncated) TTL problem \eqref{TLS:definition} is well conditioned under componentwise perturbations regardless of the choice of $s$ and $k$. Compared with the normwise condition number, the mixed and componentwise condition numbers may capture the  true conditioning of this TTLS problem. We also observe  that the SCE-based condition estimations can provide reliable estimations. Moreover,
numerical results of the untruncated mixed and componentwise condition numbers $m(A,\bsb)$ and $c(A,\bsb)$ given in Theorem \ref{th:mixed} are equal to corresponding values of the untruncated ones $m_1(A,\bsb)$ and $c_1(A,\bsb)$ given by \eqref{eq:mixed}--\eqref{eq:mixed2}, respectively. The similar conclusion can be drawn for comparing  $\kappa^{rel}(A,\bsb)$ and $\kappa^{rel}_1(A,\bsb)$ for $k=2$. 

\begin{table}\centering
\caption{\label{table:small}\small Comparisons of exact normwise, mixed and componentwise condition numbers with SCE-based condition estimations  for Example \ref{example:small} via Algorithms \ref{algo:subnorm} and \ref{eq:es mixed} with $\ell=3$.
}
{\scriptsize
\begin{tabular}{|ccccccccccc|}\hline
$s$&$k$&$\kappa^{rel}(A,\bsb)$ &$\kappa^{rel}_1(A,\bsb)$ & $\kappa_{\rm SCE}^{{\rm TTLS},(\ell)}$ &$m(A,\bsb)$ &$m_1(A,\bsb)$  & $m_{\rm SCE}^{\mathrm{TTLS},(\ell)}$ & $c(A,\bsb)$&$c_1(A,\bsb)$ & $c_{\rm SCE}^{\mathrm{TTLS},(\ell)}$ \\
\hline
$3$&$1$& $1.18\cdot 10^4$ &$\times$& $1.15\cdot 10^4$ & $4.50$&$\times$& $2.70$ & $16.20$ &$\times$& $11.65$ \\

3$$&$2$& $4.11 \cdot 10^3$&$4.11 \cdot 10^3$ & $5.46\cdot 10^3$ & $3.33$& $3.33$& $1.40$ & $4.50$ &$4.50$& $2.19$ \\
\hline
$6$&$1$& $1.18\cdot 10^7$&$\times$ & $1.51\cdot 10^7$ & $4.50$&$\times$& $3.02$ & $16.05$&$\times$ & $10.16$  \\

$6$&$2$& $4.11 \cdot 10^6$ &$4.11 \cdot 10^6$& $5.67\cdot 10^6$ & $3.33$& $3.33$& $2.35$ & $4.50$ &$4.50$ & $2.35$  \\
\hline
$9$&$1$& $1.18\cdot 10^{10}$&$\times$ & $1.42\cdot 10^{10}$ & $4.50$& $\times$ & $2.31$ & $4.50$&$\times$& $2.31$ \\

$9$&$2$& $4.11\cdot 10^{9}$&$4.11\cdot 10^{9}$ & $2.98\cdot 10^{9}$ & $3.33$&$3.33$& $2.43$ &$4.50$&$4.50$& $3.69$ \\
\hline
$12$&$1$& $1.18\cdot 10^{13}$&$\times$ & $1.61\cdot 10^{13}$ & $4.50$& $\times$ & $3.37$ & $4.50$&$\times$& $3.37$ \\

$12$&$2$& $4.11\cdot 10^{12}$&$4.11\cdot 10^{12}$ & $3.83\cdot 10^{12}$ & $3.33$&$3.33$& $1.20$ &$4.50$&$4.50$& $3.17$ \\
\hline
\end{tabular}
}
\end{table}

\begin{example}\label{example:toeplitzk6} \protect\cite{VanHuffelVandewalle1991Book}
\em 
Let the data matrix $A$ and the observation vector $\bsb$ be given by
$$A=\left[\begin{matrix} m-1&-1&\cdots&-1\\
               -1&m-1&\cdots&-1\\
               \vdots&\vdots&\ddots&\vdots\\
               -1&-1&\cdots&m-1\\
               -1&-1&\cdots&-1\\
               -1&-1&\cdots&-1
\end{matrix}\right]\in \R^{m \times (m-2)},\quad
\bsb=\left[\begin{matrix}-1\\-1\\\vdots\\m-1\\-1
\end{matrix}
\right]\in \R^{m}.
$$
Since the first $m$-$2$ singular values of the augmented matrix $[A~\bsb]$ are equal and larger than the $(m-1)$-th singular value $\sigma_{m-1}$, the truncated level $k$ can only be $m-2$. It is clear that $A$ is a Toeplitz matrix.  
\end{example}

A condition estimation is said to be reliable if  the estimations fall within one tenth to ten times of the corresponding exact condition numbers (cf. \cite[Chapter 15]{Higham2002Book}).  Table \ref{table:Toeplitzsmall} displays the numerical results for Example \ref{example:toeplitzk6} by choosing from $m=100$ to $m=500$. From Table \ref{table:Toeplitzsmall}, we can conclude that Algorithms \ref{algstr} can provide reliable  mixed and componentwise   condition estimations for this specific Toeplitz matrix $A$ and the observation vector $\bsb$,  while the normwise condition estimation  may seriously overestimate the true relative normwise condition number. We also see that the unstructured mixed and componentwise condition numbers $m(A,\bsb)$, $c(A,\bsb)$ given by Theorem \ref{th:mixed} are not smaller than the corresponding structured ones $m_s(A,\bsb)$, $c_s(A,\bsb)$ shown in Theorem \ref{th:mixed s}, which is consistent with Proposition \ref{pro:relation of the unstru and the stru}. Numerical values of the structured normwise, mixed and componentwise condition number are smaller than the corresponding counterparts.

\begin{table}\centering
\caption{\label{table:Toeplitzsmall}\small Comparisons of true structured normwise, mixed and componentwise condition numbers with  SCE-based condition estimations  for Example \ref{example:toeplitzk6} via Algorithms \ref{algstr} with the truncated level $k=m-2$ and $\ell=3$.}
{\scriptsize
\begin{tabular}{|cccccccccc|}\hline
$m$&$\kappa^{rel}(A,\bsb)$ &$\kappa_s^{rel}(A,\bsb)$ & $\kappa_{\rm SCE}^{\mathrm{STTLS},(\ell)}$ &$m(A,\bsb)$ & $m_s(A,\bsb)$ & $m_{\rm SCE}^{\mathrm{STTLS},(\ell)}$ & $c(A,\bsb)$ &$c_s(A,\bsb)$ & $c_{\rm SCE}^{\mathrm{STTLS},(\ell)}$ \\
\hline
$100$& $8.98\cdot10^2$ & $9.24\cdot10^1$ & $7.37\cdot10^3$& $2.49$ & $2.49$ & $2.84$&$2.49$&$2.49$&$2.84$ \\
\hline
$200$& $1.80\cdot 10^3$ & $1.31\cdot 10^2$ & $1.93\cdot 10^4$& $2.50$ & $2.50$ & $2.27$&$2.50$&$2.50$&$2.27$ \\
\hline
$300$& $2.71\cdot 10^3$ & $1.60\cdot 10^2$ & $3.80\cdot 10^4$& $2.50$ & $2.50$ & $2.21$&$2.50$&$2.50$&$2.21$ \\
\hline
$400$& $3.61\cdot 10^3$ & $1.85\cdot 10^2$ & $5.82\cdot 10^4$& $2.50$ & $2.50$ & $2.19$&$2.50$&$2.50$&$2.19$ \\
\hline
$500$& $4.52\cdot 10^3$ & $2.06\cdot 10^2$ & $8.05\cdot 10^4$& $2.50$ & $2.50$ & $2.14$&$2.50$&$2.50$&$2.14$\\
\hline
\end{tabular}
}
\end{table}

\begin{example}\label{example:small2}
\em This test problem comes from \cite[\S 3.2]{gtt13max}. The augmented matrix $[A~\bsb] \in \R^{m\times (n+1)}$ are builded by using the SVD $[A~\bsb]=USV^{\top}$. Here, $U$ is an arbitrary orthogonal matrix with the size of $m \times m$, and $\Sigma$ be a diagonal matrix with equally spaced singular values in $[10^{-2},1]$. The matrix $V$ is generated as following:  compute the QR decomposition of  the matrix
$$
\left[\begin{matrix}
\sqrt{1-\beta^{2}}\bsc&X\\\beta \bsd&Y
\end{matrix}\right]
$$
with the Q-factor $Q$, where  $X$ and $Y$ are random matrices, $\bsc\in \R^{k}$ and $\bsd\in \R^{n+1-k}$ are normalized random vectors. Here, $k\leq \min\{m,n\}$ is truncation level. Then we set $V$ to be an orthogonal matrix that commutes the first and last rows of $Q^{\top}$. It is easy to verify that $V_{22}=\beta \bsd^{\top}$ and $\|V_{22}\|=\beta$. In this test, we take $m=400$, $n=120$, $k=80$, and $\beta=10^{-3}$. The perturbation matrices $\Delta A$ and $\Delta \bsb$ of $A$ and $\bsb$ are generated as follows:
\begin{align}\label{de:1}
 \Delta A= \epsilon\,
 (E \boxdot A), \,\,\Delta \bsb= \epsilon\, (\bsf \boxdot \bsb),
\end{align}
where $ E$ and $\bsf$ are random matrices
whose entries are uniformly distributed in the open interval $(-1,1)$, $\epsilon=10^{-8}$ represents the magnitude of the perturbation.
\end{example}


We note that   both $ \Delta A$ and $\Delta \bsb$ are componentwise perturbations on $A$ and $\bsb$ respectively. In order to illustrate the validity of these estimators  $\kappa_{\rm SCE}^{\mathrm{TTLS},(\ell)},\, m_{\rm SCE}^{\mathrm{TTLS},(\ell)}$  and
$c_{\rm SCE}^{\mathrm{TTLS},(\ell)}$ via Algorithms \ref{algo:subnorm} and
\ref{eq:es mixed}, we define
the normwise, mixed and componentwise over-estimation ratios as follows 
$$
r_\kappa=\frac{\kappa_{\rm SCE}^{\mathrm{TTLS},(\ell)} \epsilon}
{\| \tilde \bsx_k-\bsx_k\|_2/\|\bsx_k\|_2},\quad
r_m=\frac{m_{\rm SCE}^{\mathrm{TTLS},(\ell)}\epsilon}
{\| \tilde \bsx_k-\bsx_k\|_\infty/\| \bsx_k\|_\infty},\quad
r_c=\frac{c_{\rm SCE}^{\mathrm{TTLS},(\ell)}\epsilon}
{\| \frac{\tilde \bsx_k-\bsx_k}{\bsx_k} \|_\infty},
$$
Typically the ratios
in $(0.1, ~10)$ are acceptable \cite[Chap. 15]{Higham2002Book}.

\begin{figure}[htb]
\centering
\subfloat[Normwise over-estimation ratios]
{%
\includegraphics[width=.33\textwidth]{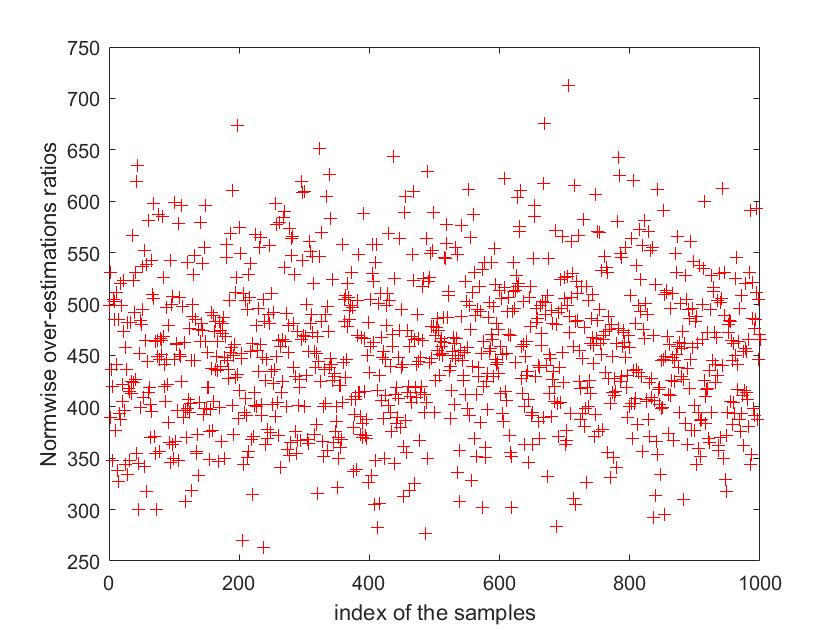}}\hfill
\subfloat[Mixed over-estimation ratios]{%
\includegraphics[width=.33\textwidth]{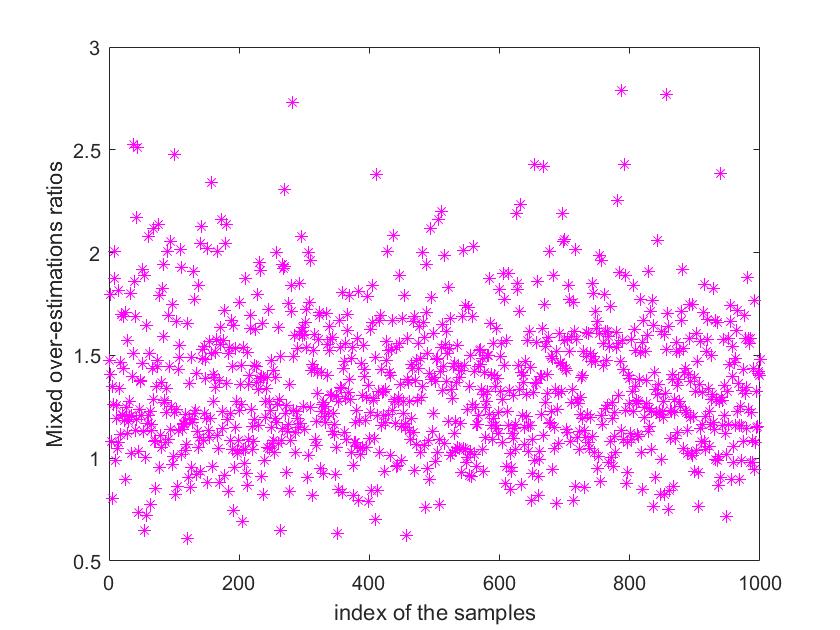}}
\subfloat[Componentwise\,over-estimation ratios]{%
\includegraphics[width=.33\textwidth]{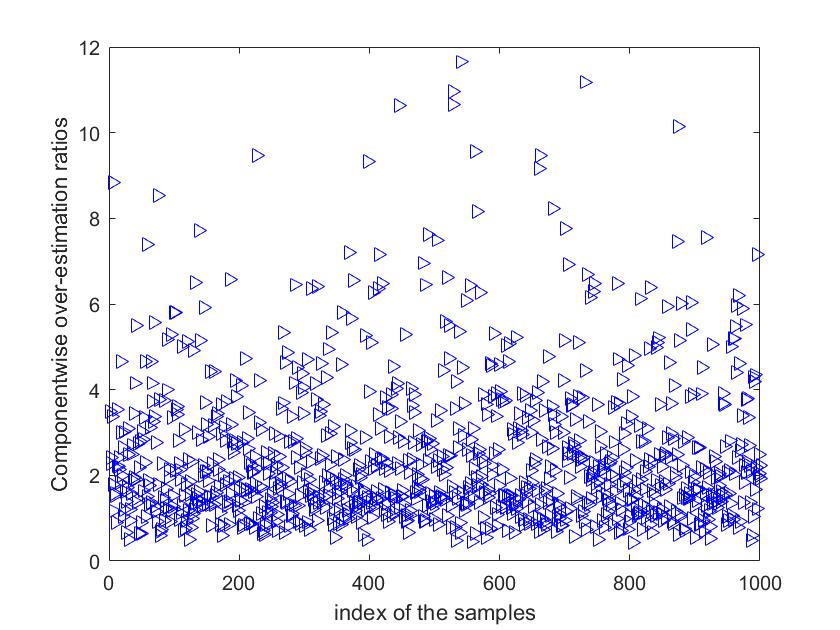}}\hfill
\caption{SCE for the TLS problem under unstructured componentwise perturbations with 1000 samples for Example \ref{example:small2}.}
\label{fig:Toep Un}
\end{figure}

Figure \ref{fig:Toep Un} displays the numerical results for Example \ref{example:small2}, where we generate 1000 random samples $[A~\bsb]$.
From Figure \ref{fig:Toep Un},  we see that $\kappa_{\rm SCE}^{\mathrm{TTLS},(\ell)}$ may seriously overestimate the true relative normwise error. All the normwise over-estimation ratios of 1000 samples is more than $10$, and the mean value of these estimations  is $453.8841$. All elements of the mixed over-estimation ratios of 1000 samples are within $(0.1,~10)$ and  the mean value of these estimations  is $1.3563$. There are only 6 entries of the componentwise over-estimation ratios are greater than 10 and the mean value of these estimations  is $2.4716$. The maximal values of the normwise, mixed and componentwise over-estimation vector are $712.5759,\,   2.7908$ and $ 11.6508$, while their corresponding minimum are $263.3766,\,0.6078,0.4261$ accordingly. Therefore the mixed and componentwise condition estimations $m_{\rm SCE}^{\rm{TTLS},(\ell)}$ and $c_{\rm SCE}^{\rm{TTLS},(\ell)}$ are reliable.

\begin{example}\label{example:toeplitz}
\em Let the Toeplitz matrix $A$ and the vector $\bsb$ are defined in Example \ref{example:toeplitzk6}, where $m=500$. We generate  1000 structured componentwise perturbations $\Delta A_1= \epsilon \,( E_1 \boxdot A)$ and 1000 unstructured componentwise perturbations $\Delta A_2= \epsilon \,( E_2 \boxdot A)$, where $ E_1$ is a random Toeplitz matrix and $ E_2$ is a random matrix whose entries are uniformly distributed in the open internal $(-1,1)$. And $\Delta \bsb= \epsilon \,(f \boxdot \bsb)$, where $\bsf$ is a random matrix with components uniformly distributing in the open interval $(-1,1)$.
\end{example}

For Example  \ref{example:toeplitz}, we use $r_\kappa^S$, $r_m^S$ and $r_c^S$ to  denote the structured normwise, mixed and componentwise over-estimation ratios corresponding to structured componentwise perturbations of $\Delta A_1$ and $\Delta \bsb$,  and $r_\kappa$, $r_m$ and $r_c$ are unstructured normwise, mixed and componentwise over-estimation ratios corresponding to unstructured componentwise perturbations of $\Delta A_2$ and $\Delta \bsb$, where
\[
r_\kappa^{\rm S}=\frac{\kappa_{\rm SCE}^{\mathrm{STTLS},(\ell)}\, \epsilon}
{\|\tilde \bsx_k - \bsx_k\|_2/\|\bsx_k\|_2},\quad
r_m^{\rm S}=\frac{m_{\rm SCE}^{\mathrm{STTLS},(\ell)}\, \epsilon}
{\|\tilde \bsx_k - \bsx_k\|_\infty/\| \bsx_k\|_\infty},\quad
r_c^{\rm S}=\frac{c_{\rm SCE}^{\mathrm{STTLS},(\ell)}\, \epsilon}
{\| \frac{\tilde \bsx_k - \bsx_k}{\bsx_k} \|_\infty},
\]
$$
r_\kappa=\frac{\kappa_{\rm SCE}^{\mathrm{TTLS},(\ell)} \epsilon}
{\|\tilde \bsx_k - \bsx_k\|_2/\|\bsx_k\|_2},\quad
r_m=\frac{m_{\rm SCE}^{\mathrm{TTLS},(\ell)} \epsilon}
{\|\tilde \bsx_k - \bsx_k\|_\infty/\| \bsx_k\|_\infty},\quad
r_c=\frac{c_{\rm SCE}^{\mathrm{TTLS},(\ell)} \epsilon}
{\|  \frac{\tilde \bsx_k - \bsx_k}{\bsx_k} \|_\infty}.
$$

Figure \ref{fig:Toepstru} displays  the numerical results for Example \ref{example:toeplitz} with $\ell=3$ and $\epsilon=10^{-8}$. Here,  in Figure \ref{fig:Toepstru}(A)--Figure \ref{fig:Toepstru}(C), the symbol ``+" in the blue color denote the numerical values of $r_\kappa^{\rm S}$, $r_m^{\rm S}$ and $r_c^{\rm S}$  corresponding to 1000 structured perturbations 
 while  the symbol ``*" in the red color denote the numerical values of $r_\kappa$, $r_m$ and $r_c$ corresponding to 1000 unstructured perturbations. 

From Figure \ref{fig:Toepstru}, we observe that  the mixed and componentwise condition estimations $m_{\rm SCE}^{\mathrm{STTLS},(\ell)}$, $c_{\rm SCE}^{\mathrm{STTLS},(\ell)}$, $m_{\rm SCE}^{\mathrm{TTLS},(\ell)}$ and $c_{\rm SCE}^{\mathrm{TTLS},(\ell)}$ are reliable, while the structured normwise condition estimation $\kappa_{\rm SCE}^{\mathrm{STTLS},(\ell)}$ may seriously over-estimate the true relative normwise error. Furthermore, we can also conclude  that the over-estimation ratios associated with the structured mixed and component condition numbers are smaller than the unstructured counterparts  in most cases, which are consistent with  the conclusion in Proposition \ref{pro:relation of the unstru and the stru}. The mean values of $r_m^{\rm S}$, $r_m$, $r_c^{\rm S}$, and $r_c$ of $1000$ samples are  $1.7140$, $2.4642$, $1.7140$, and $2.4642$ respectively.  Moreover, all these unstructured and structured mixed and componentwise condition over-estimation ratios  are with $(0.1,10)$, which indicate mixed and componentwise condition estimations are reliable.


\begin{figure}[htb]
\centering
\subfloat[Unstructured and structured normwise over-estimation ratios]
{%
\includegraphics[width=.33\textwidth]{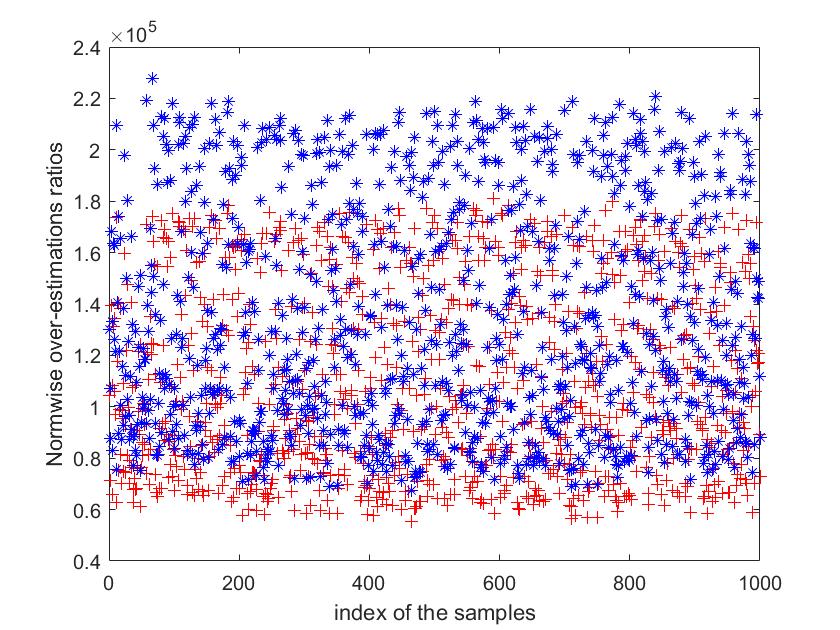}}\hfill
\subfloat[Unstructured and structured mixed over-estimation ratios]{%
\includegraphics[width=.33\textwidth]{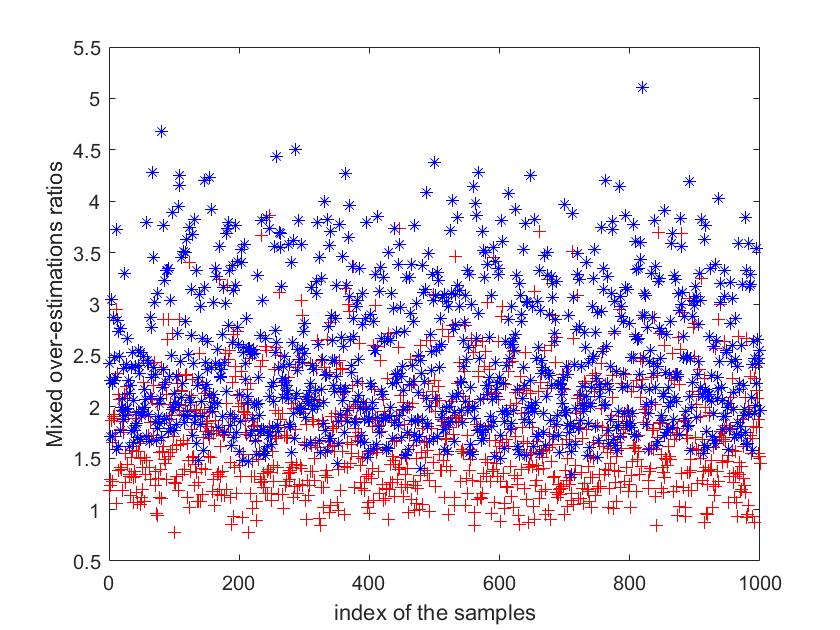}}
\subfloat[Unstructured and structured componentwise over-estimation ratios]{%
\includegraphics[width=.33\textwidth]{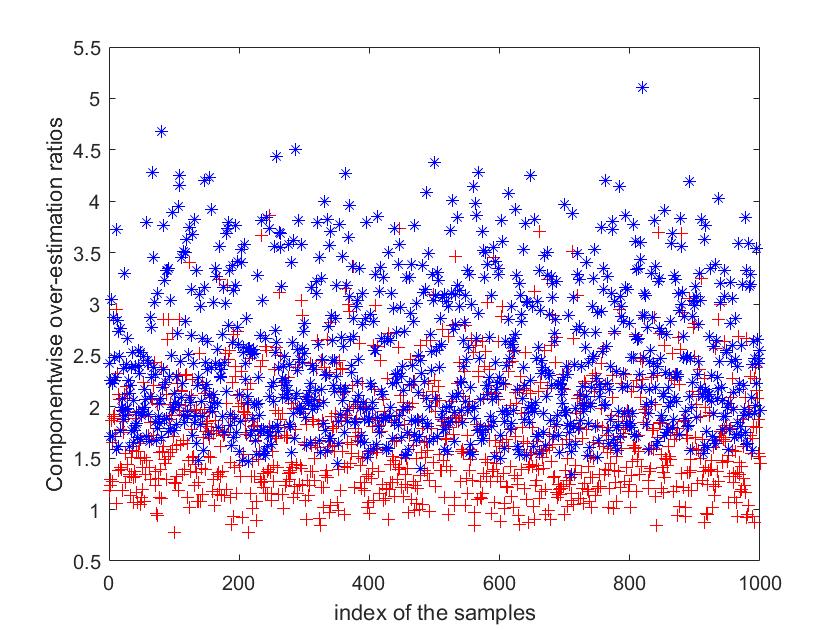}}\hfill
\caption{\small{SCE for the TLS problem under unstructured and structured componentwise perturbations with 1000 unstructured and structured  perturbations for Example \ref{example:toeplitz}}.}
\label{fig:Toepstru}
\end{figure}


\begin{example}\label{example:5}
\em
This example comes from the model that restructures the image named Shepp-Logan ``head phantom'' (Shepp and Logan 1974) by the TTLS technique, which is widely used in inverse scattering studies. In fact, the TTLS method has been used to study the ultrasound inverse scattering imaging \cite{Chao2003A}. Here, we utilize the $\textsc{MATLAB}$ file ``paralleltomo.m'' from the $\textsc{testprobs}$ suite\footnote{Netlib: http://www.netlib.org/numeralgo/ or GitHub: https://github.com/jakobsj/AIRToolsII.} to create parallel-beam CT test problem and obtain the exact phantom. The input parameters of ``paralleltomo.m'' are set to be $N=40$, $\theta = 0:5:175$, and $p = 55$, hence we can obtain  a $1834$-by-$1600$  matrix $A$ and $1834$-by-$1$ right-hand vector $\bsb$. The 500 perturbations $\Delta A$ and $\Delta \bsb$ of $A$ and $\bsb$ are generated  as in \eqref{de:1}.
\end{example}

Table \ref{table:phantom} lists  the condition estimations $\kappa_{\rm SCE}^{\mathrm{STTLS},(\ell)}$, $c_{\rm SCE}^{\mathrm{TTLS},(\ell)}$ and $m_{\rm SCE}^{\mathrm{TTLS},(\ell)}$ with truncation level $k=1536$ and the corresponding relative errors with respect to  different  magnitude  $\epsilon$ of the perturbation for Example \ref{example:5}. From Table \ref{table:phantom}, we can see that the relative errors are bounded by the product of $\epsilon $ and corresponding condition estimations, which means that the proposed condition estimations can  give reliable  error bounds.  The componentwise condition estimation $c_{\rm SCE}^{\mathrm{TTLS},(\ell)}$ is much larger than
$\kappa_{\rm SCE}^{\mathrm{STTLS},(\ell)}$ and $m_{\rm SCE}^{\mathrm{TTLS},(\ell)}$ since the minimum absolute component of the TTLS solution $\bsx_k$ is too small, which is order of $10^{-6}$. Hence the component of the TTLS solution $\bsx_k$ in the sense of the tiny magnitude  is very sensitive to small perturbations on the underlying component of $\bsx_k$. 

%
%
%
%

\begin{table}\centering
\caption{\label{table:phantom}\small Comparison of the relative errors and SCE-based estimations by Algorithms 1 and 2 under 500 perturbations with different perturbation magnitudes for Example \ref{example:5}.}
{\tiny
\begin{tabular}{ccccccccc}\hline
$\epsilon$&$10^{-1}$ &$10^{-2}$ & $10^{-3}$ & $10^{-4}$& $10^{-5}$& $10^{-6}$& $10^{-7}$& $10^{-8}$  \vspace{0.05cm}\\
\hline
$\frac{\left\|\tilde \bsx_k -\bsx_k\right\|_{2}}{\|\bsx_k\|_{2}}$&$ 3.25\cdot 10^{-2}$ &$ 3.25\cdot 10^{-3}$ & $ 3.26\cdot 10^{-4}$ & $3.26\cdot 10^{-5}$ & $3.26\cdot 10^{-6}$ & $3.26\cdot 10^{-7}$& $ 3.26\cdot 10^{-8}$& $3.26\cdot 10^{-9}  $\vspace{0.1cm}\\
$\kappa_{\rm SCE}^{\mathrm{TTLS},(\ell)}$&$ 3.33\cdot 10^{4}$ &$ 3.33\cdot 10^{4}$  &$3.33\cdot 10^{4}$ &$ 3.33\cdot 10^{4}$ &$ 3.33\cdot 10^{4}$ &$ 3.333\cdot 10^{4}$ &$ 3.33\cdot 10^{4}$ &$ 3.33\cdot 10^{4}$    \vspace{0.05cm}\\
\hline
$\frac{\left\|\tilde \bsx_k -\bsx_k\right\|_{\infty}}{\|x_k\|_{\infty}}$&$ 2.62\cdot 10^{-2}$ &$ 2.82\cdot 10^{-3}$ & $ 2.89\cdot 10^{-4}$ & $ 2.90\cdot 10^{-5}$& $ 2.90\cdot 10^{-6}$&$ 2.90\cdot 10^{-7}$&$ 2.90\cdot 10^{-8}$ &$ 2.90\cdot 10^{-9}$\vspace{0.05cm}\\
$m_{\rm SCE}^{\mathrm{TTLS},(\ell)}$&$ 3.74\cdot 10^{1}$ &$ 3.74\cdot 10^{1}$  &$ 3.74\cdot 10^{1}$ &$3.74\cdot 10^{1}$ &$3.74\cdot 10^{1}$ &$ 3.74\cdot 10^{1}$ &$ 3.74\cdot 10^{1}$ &$ 3.74\cdot 10^{1}$    \vspace{0.05cm}\\
\hline
$\left\|\frac{\tilde \bsx_k -\bsx_k}{\bsx_k}\right\|_{\infty}$&$ 3.65\cdot 10^{2}$ &$ 7.79\cdot 10^{1}$ & $ 8.38\cdot 10^{0}$ & $ 8.44\cdot 10^{-1}$& $ 8.44\cdot 10^{-2}$&$ 8.44\cdot 10^{-3}$&$ 8.44\cdot 10^{-4}$&$ 8.44\cdot 10^{-5}$\vspace{0.05cm}\\
$c_{\rm SCE}^{\mathrm{TTLS},(\ell)}$&$1.85\cdot 10^{6}$ &$1.85\cdot 10^{6}$&$1.85\cdot 10^{6}$&$1.85\cdot 10^{6}$&$1.85\cdot 10^{6}$&$1.85\cdot 10^{6}$&$1.85\cdot 10^{6}$&$1.85\cdot 10^{6}$\vspace{0.05cm}\\
\hline
\end{tabular}
}
\end{table}

\section{\bf Concluding remarks}\label{sec:con}

In this paper, we study the mixed and componentwise condition numbers of the TTLS problem under the genericity condition. We also consider the structured condition estimation for the STTLS problem and investigate the relationship between the unstructured condition numbers and the corresponding  structured counterparts. When the TTLS problem degenerates the untrunctated TLS problem, we prove that condition numbers for the TTLS problem can recover the previous condition numbers for the TLS problem from their explicit expressions. Based on SCE,   normwise, mixed and componentwise condition estimations algorithms are proposed for the TTLS problem, which can be integrated into the SVD-based direct solver for the TTLS problem. Numerical examples indicate that, in practice, it is better to adopt the componentwise perturbation analysis for the TTLS problem and the proposed algorithms are reliable, which provide posterior error estimations of high accuracy.  The results in this paper can be extended to the truncated singular value solution of a linear ill-posed problem \cite{BergouThe}. We will report our progresses on the above topic elsewhere in the future.


%

\end{document}